\documentclass[12pt,
  oneside]{amsart}

\usepackage{amscd, amsmath, colonequals, cite}

\usepackage[usenames]{xcolor}
\usepackage{tikz}\usetikzlibrary{positioning,decorations.markings, patterns}
\usepackage{pgfplots}\pgfplotsset{compat=1.15}
\usepackage{caption, subcaption}

\title[]{A survey on normal forms of real submanifolds with CR singularity}
\author[]{Xianghong Gong$^{\dagger}$}

\address{Department of Mathematics,
 University of Wisconsin, Madison, WI 53706, U.S.A.}
 \email{gong@math.wisc.edu}

\author{Laurent Stolovitch}
\address{CNRS and Laboratoire J.-A. Dieudonn\'e
U.M.R. 6621, Universit\'e de Nice - Sophia Antipolis, Parc Valrose
06108 Nice Cedex 02, France.}
\email{stolo@unice.fr}
\thanks{$^{\dagger}$Partially supported by  NSF grant  DMS-2349865.
}


\newcommand{\Fix}{\operatorname{Fix}}
\newcommand{\Diff}{\mathrm{Diff}}
\newcommand{\hatDiff}{\widehat{\Diff}}
\newcommand{\bX}{\bf X}
\newcommand{\tdd}[1]{\tfrac{\partial}{\partial #1 \vphantom{\hat X}}}

\newtheorem{thm}{Theorem}[section]
\newtheorem{cor}[thm]{Corollary}
\newtheorem{prop}[thm]{Proposition}
\newtheorem{lemma}[thm]{Lemma}
\newtheorem{remark}[thm]{Remark}
\newtheorem{assumption}[thm]{Assumptions}

\newcommand{\hot}{\mathrm{h.o.t.}}
\theoremstyle{definition}

\newtheorem{defn}[thm]{Definition}
\newtheorem{exmp}[thm]{Example}

\newtheorem{rem}[thm]{Remark}

\newtheorem{conj}[thm]{Conjecture}
\newtheorem{problem}[thm]{Problem}

\renewcommand{\th}[1]{\begin{thm}\label{#1}}
\newcommand{\eth}{\end{thm}}
\newcommand{\co}[1]{\begin{cor}\label{#1}}
\newcommand{\eco}{\end{cor}}
\renewcommand{\le}[1]{\begin{lemma}\label{#1}}
\newcommand{\ele}{\end{lemma}}
\newcommand{\pr}[1]{\begin{prop}\label{#1}}
\newcommand{\epr}{\end{prop}}

\newcommand{\ga}{\begin{gather}}
\newcommand{\ega}{\end{gather}}
\newcommand{\gan}{\begin{gather*}}
\newcommand{\egan}{\end{gather*}}
\newcommand{\al}{\begin{align}}
\newcommand{\eal}{\end{align}}
\newcommand{\aln}{\begin{align*}}
\newcommand{\ealn}{\end{align*}}
\newcommand{\eq}[1]{\begin{equation}\label{#1}}
\newcommand{\eeq}{\end{equation}}

\newcommand{\ci}{~\cite}

\newcommand{\f}[2]{\frac{#1}{#2}}

\newcommand{\dd}[2]{\frac{\partial #1}{\partial #2}}

\newcommand{\cc}{{\mathbb C}}
\newcommand{\nn}{{\mathbb N}}
\newcommand{\zz}{{\mathbb Z}}
\newcommand{\rr}{{\mathbb R}}
\newcommand{\Tt}{{\bf T}}
\newcommand{\qq}{{\bf Q}}

\newcommand{\ov}{\overline}

\newcommand{\id}{\operatorname{I}}

\newcommand{\RE}{\operatorname{Re}}
\newcommand{\IM}{\operatorname{Im}}

\newcommand{\cL}{\mathcal}

\newcommand{\all}{\alpha}

\newcommand{\gaa}{\gamma}

\newcommand{\del}{\delta}
\newcommand{\Del}{\Delta}
\newcommand{\var}{\varphi}
\newcommand{\e}{\epsilon}
\newcommand{\om}{\omega}
\newcommand{\Om}{\Omega}

\newcommand{\la}{\lambda}

\newcommand{\pd}{\partial}

\newcommand{\re}[1]{(\ref{#1})}
\newcommand{\rea}[1]{$(\ref{#1})$}

\newcommand{\nrc}[1]{Corollary~\ref{#1}}
\newcommand{\rp}[1]{Proposition~\ref{#1}}
\newcommand{\rt}[1]{Theorem~\ref{#1}}

\newcommand{\rta}[1]{Theorem~$\ref{#1}$}

\newcounter{pp}
\newcommand{\bpp}{\begin{list}{$\hspace{-1em}\alph{pp})$}{\usecounter{pp}}}
\newcommand{\epp}{\end{list}}

\newcounter{ppp}
\newcommand{\bppp}{\begin{list}{$\hspace{-1em}(\roman{ppp})$}{\usecounter{ppp}}}
\newcommand{\eppp}{\end{list}}

\def\beq{\begin{equation}}
\def\eeq{\end{equation}}

\begin{document}


\begin{abstract}
We survey results dating back from the seminal works of Bishop and Moser-Webster as well as more recent advances.
\end{abstract}

\date{\today}
 \maketitle
\tableofcontents


%


\setcounter{section}{0}
\setcounter{thm}{0}\setcounter{equation}{0}
\section{Introduction}
\label{sect1}

Let $M$ a real analytic submanifold of $\cc^n$. 
We first describe some local holomorphic invariants of $M$.
Denote by $T_xM$  the real tangent space of $M$ at $x$ and $T_x^{(1,0)}M$ the complex tangent space consisting of vectors $\sum c_j\frac{\pd}{\pd z_j}$ that are tangent to $M$ at $x$.   When  
$x\mapsto d_x:=\dim T_x^{(1,0)}M$ is a constant function on $M$,
$M$ is called a  Cauchy-Riemann (CR) submanifold. We say that a point $x_0\in M$ is  CR singular or a  complex tangent, 
 if $d_x$ is not  constant  in any neighborhood of $x_0$; let $M_{CRs}$ denote the set of complex tangents of $M$. Thus, a real hypersurface is always a CR submanifold.
 This survey is concerned  with the local holomorphic classification of real analytic submanifolds in $\cc^n$. One way to approach such a classification is to find a formal classification of real submanifolds under formal holomorphic mappings,  that are  formal power series in complex variables and then study the convergence of the mappings. 
The normal form of a real analytic hypersurface at a Levi non-degenerate point was achieved by the cerebrated works of \'E.~Cartan~\cite{Ca32,Ca33},  Tanaka~\cite{Ta62}, and Chern-Moser~\cite{chern-moser}. Chern and Moser constructed a formal normal form first and then showed that the normal form can be achieved by a (convergent) local biholomorphic map.   There are many important works about normal forms of real analytic hypersurfaces since the work of Chern-Moser, which study the Levi-degenerate case. A problem that has attracted a lot of attention is to study if two real analytic real hypersurfaces in $\cc^n$ are holomorphically equivalent if they are formally equivalent. Positive solutions were obtained depending on the types of degeneracy of the Levi forms. 
Baouendi, Ebenfelt and Rothschild~\ci{BER97}
 investigated  the holomorphic equivalence of real analytic CR of submanifolds in $\mathbb C^n$.  They showed that the formal equivalence of
  two real analytic hypersurfaces that are finitely non-degenerate indeed implies  their holomorphic equivalence.
 Juhlin-Lamel~\ci{JL13} proved convergence of formal equivalences  for 1-nonminimal hypersurfaces in $\cc^2$.
Finally,  Kossovskiy-Shafikov~\ci{KS16}  proved 
that there are two germs of
real analytic smooth  hypersurfaces of $m$-nonminimal ($m\geq2$) in $\mathbb C^2$ that are formally but
not holomorphically equivalent. Lamel and Stolovitch~\cite{lamel-stolo} gave a higher co-dimensional version of Chern-Moser theory. They defined
 an appropriate notion of normal form as well as conditions ensuring the holomorphy of a transformation  to such a normal form.

To limit the scope of this survey, we  would like  to focus on real submanifolds that have CR singularity; even within this scope we will mainly discuss topics in connections with the normal form theory. The normal-form results in this direction are already rich,  and they show solid connections between several complex variables and  dynamical systems, for instance. We will also present at the end of survey some open problems.

 A real submanifold in $\cc^n$ with a CR singularity must have codimension at least $2$.
The study of real submanifolds with CR singularity was initiated by E.~Bishop in his pioneering work~\cite{Bi65}. He investigated  a $C^\infty$ real submanifold $M$ of which
the complex tangent space  at a complex tangent
has dimension $1$, the smallest possible value.
 The very basic models of such submanifolds   are    the Bishop quadrics
$$
Q_\gaa\subset\cc^n\colon z_n=z_1\ov z_1+\gaa(z_1^2+\ov z_1^2),\quad \IM z_2=\cdots=\IM z_{n-1}=0,
$$
where $0\leq\gamma\leq\infty$ is the
 Bishop invariant, and $Q_\infty$ is defined by
$x_n=z_1^2+\ov z_1^2$, $\IM z_2=\cdots=\IM z_{n}=0$.
According to Bishop, the complex tangent at the origin  is  {\it elliptic} if $0\leq\gaa<1/2$,  {\it parabolic} if $\gaa=1/2$, or {\it hyperbolic} if $\gaa>1/2$.

In~\cite{MW83}, Moser and Webster studied the normal form problem of   real analytic $n$-submanifolds
$M$ in $\cc^n$ at a complex tangent with $\gamma\neq0$. They discovered that when $\gamma$ is non zero, the $M$ is associated with a  triple $\{\tau_1,\tau_2,\rho\}$ defined in a neighborhood of the origin of complexification $\mathcal M$ of $M$, where $\mathcal M$ can be locally identified with $\cc^n$, and $\tau_1,\tau_2$ are holomorphic involutions, i.e. $\tau_j^2=I$, and $\rho$ is an antiholomorphic involution, and $\tau_2=\rho\tau_1\rho^{-1}$. The significance of this discovery is that the local holomorphic (resp. formal holomorphic) classification of Bishop type $n$-submanifolds (with $\gamma\neq0$) in $\cc^n$ is equivalent to the local holomorphic (resp. formal holomorphic)  classification of the set of $\{\tau_1,\tau_2,\rho\}$.  The composition $\sigma=\tau_2\tau_1$ is reversible, i.e. $\sigma^{-1}=\tau_1\sigma\tau_1^{-1}$ with $\tau_1^2=I$. Such mappings in the real category  have been studied extensively in dynamical systems. For instance, Birkhoff~\cite{Bi15,Bi66} established the existence of periodic orbits for such planar mappings along with area-preserving mappings. The reader is referred to~\ci{sevryuk-lnm} for reversible dynamics.

%
%

The survey is organized as follows. In Section 2, we give a brief introduction on the Bishop invariant and the Moser-Webster normal form theory. In Section~\ref{sec:elli}, we present the results of Moser-Webster on surfaces with non vanishing elliptic Bishop invariant. 

In Section~\ref{gamma=0}, we describe results on real analytic surfaces with vanishing Bishop invariant. The normalization for this case is quite different from the case of non-vanishing Bishop invariant since this is the case where the Moser-Webster involution theory breaks down. 

In Section~\ref{gamma>.5}, we describe results on real analytic surfaces with   hyperbolic CR singularity. We will first discuss the case where the hyperbolic Bishop invariant is non-exceptional, i.e. the Moser-Webster invariant $\mu$ is not a root of unity. When $\mu$ is a root of unity, we will present a KAM type result on the existence of a large family of complex invariant curves that intersecting the real surfaces in real curves.  

In Section~\ref{sect:exec}, we describe results for holomorphically flattened surfaces in $\cc^2$ with the exceptional hyperbolic complex tangent. Among other technics, we will illustrate how the Écalle-Voronin moduli space for $1$-dimensional holomorphic maps can be used for the study of a reversible 2-dimensional map of which the linear part is periodic.  

In Section~\ref{gamma=.5}, we present the moduli space for real analytic $n$-manifolds in $\cc^n$ that are formally equivalent to $Q_{1/2}$. It turns out the classification has an infinite-dimensional moduli space, via an application of Voronin's moduli space for holomorphic mappings of $\cc^n$ with $n>1$.

In Section~\ref{sect:highcod}, we survey results on local hull of holomorphy for codimension two submanifolds with CR singularity.  

In Section~\ref{sect:lowdim}, we present a result on normalization of real analytic $m$-submanifolds with CR singularity in $\cc^n$ with $2(n+1)/3\leq m<n$.

The results from Section 2 through Section~\ref{gamma=.5} are concerned with 
$n$ submanifolds in $\mathbb C^n$  whose complex tangent space at the CR singular point is {\it minimum}, i.e.  one-dimensional.
In Section~\ref{sect:maxCR}, we discuss real analytic $n$-submanifolds $M$ in $\cc^n$ whose complex tangent space at the CR singular point has the maximum $n/2$ dimension with $n\geq 4$ an even integer.   For this maximum CR singularity, we describe and abelian CR singularity and its convergence of normalization under a Diophantine-type condition. Then we  use the normal form for abelian CR singularity to find the local hull of holomorphy when the complex tangents are elliptic. 
We also present results on the existence of intersecting submanifolds for hyperbolic complex tangents. Finally we present a result on divergent Poincar\'e-Dulac normal form for elliptic CR singularity in $\cc^6$.

In Section~\ref{sect:openprob}, we list some open problems.

\section{CR-singular  $n$-submanifolds in $\cc^n$}
\setcounter{thm}{0}\setcounter{equation}{0}
\subsection{Notation} 
Let us first introduce notation used throughout the paper.
We say that two real analytic submanifolds $M,\tilde M$ in $\cc^n$ are equivalent, if $M,\tilde M$ contain the origin and there is a biholomorphic mapping $\Phi$ of $\cc^n$ defined is a  neighborhood of the origin such that $\Phi(0)=0$ and $\Phi(M)=\tilde M$ as  germs of real analytic submanifold. We say that $M,\tilde M$ are formally equivalent, if $\Phi(0)=0$, $\Phi'(0)$ is non-singular, and $\Phi(z)$ is given by formal power series in $z$, and $\Phi(M)=\tilde M$ as formal submanifolds at $0$. A formal power series $f(z,\bar z)$ of the sole {\it complex} variable $z$ might be referred to {\it formal holomorphic} power series.

\subsection{Bishop invariant $\gamma$} 
Let $M$ be a real analytic $n$-submanifold in $\cc^n$ that has a complex tangent at the origin. Suppose that the dimension   of $T_0^{(1,0)}M$ is one, the smallest possible value. By a linear change of coordinates, we may assume that $T_0M$ is given by $z_n=0$ and $y':=(y_2,\dots, y_{n-1})=0$. Thus
\begin{equation}\label{eq:M1}
	M\subset \cc^n: z_n=F_n(z_1,\bar z_1,x'),\quad y_\all=F_\all(z_1,\bar z_1,x'),\quad 1< \all<n
\end{equation}
where  $F(z_1,\bar z_1,x')=O(2)$ and $F_\all(z_1,\ov z_1,x')=O(2)$ are respectively complex-valued and real-valued power series in $z_1,\bar z_1,x'$ without constant or linear terms. Since $F_\all$ are  real-valued, we have
\eq{barF_a}
\ov F_\all(\bar z_1,z_1,x'):=\ov{F_\all(z_1,\bar z_1,x')}=F_\all(z_1,\ov z_1,x')=O(2).
\eeq
 The set $M_{CRs}$ of complex tangents is the zero set of $\omega=dz_1\wedge\cdots\wedge dz_n|_M$. We have $\om=C(z_1,\ov z_1,x')dz_1\wedge d\bar z_1\wedge dx_2\wedge\cdots dx_{n-1}$. We assume that the complex tangent point $0$ is {\it non-degenerate}, i.e. $C(z_1,\bar z_1,0)$ is not identically zero, i.e. $dC\neq0$ on $T_0^{(1,0)}M$. Changing coordinates to simplify the quadratic term in $F_n$, we can get
\eq{Fn}
F_n(z_1,\bar z_1,y')=q_\gamma(z_1,\bar z_1)+i\del_n(z_1+\bar z_1)x_{2}+O(3)\eeq
where $\delta_n=0$   and $\del_n=0,1$ for $n>2$, and
 \eq{qgamma}
 q_\gamma(z_1,\bar z_1)= z_1\ov z_1+\gaa(z_1^2+\ov z_1^2),\quad 0\leq\gamma<\infty; \qquad q_\infty(z_1,\bar z_1)=z_1^2+\bar z_1^2.\eeq
Note that
$$C(z_1,\bar z_1,x')=z_1+2\gamma \bar z_1+i\delta_n x_{n-1}+O(2).$$
Thus for $\gamma\neq1/2$, $M_{CRs}$ is a codimension-two submanifold of $M$. When $\gamma_n=1/2$, one can see that $M_{CRs}$ can be a real analytic subset of the real hypersurface $\RE C=0$ in $M$ and $\dim_{CRs}$ can be any number between $0$ and $n-1$.
   When the dimension is $n-1$, we must have $\del_n=0$ and the complex tangents near the origin remain to be parabolic since the set of complex tangents near a hyperbolic or elliptic complex tangent has codimension two in $M$.

\begin{rem}One can see that  $M_{CRs}$, defined by $C=0$, is a closed subset of $M$ and the $\dim T_x^{(1,0)}M$ is upper semi-continuous on $M_{CRs}$. The set of non-degenerate complex tangents $x$, defined as the set of $x$ such that    $dC$ is not identically $0$  on $T_x^{(1,0)}M$, is    an open subset of $M_{CRs}$. 
\end{rem}

\subsection{Moser-Webster involutions} 
A useful   technic to study real analytic submanifolds is  to complexify   submanifolds and holomorphic mappings. Let us describe the complexification in our situation.  Let
\begin{equation}\label{eq:Mcomplex}
	\mathcal M\subset\cc^n\times\cc^n:\begin{cases}
 z_n=F_1(z_1,w_1,\frac{z'+w'}{2}),\quad w_n=\bar F_n(w_1,z_1, \frac{z'+w'}{2}), \\ z_j-w_j=2iF_j(z_1,w_1,\frac{z'+w'}{2})
 \end{cases}
\end{equation}
be the \emph{complexification} of $M$, where $(z,w)$ are coordinates of $\cc^n\times\cc^n$.
The antiholomorphic involution $\rho(z,w)=(\bar w,\bar z)$
preserves $\cL M$; via embedding $z\to (z,\ov z)$, we have
$M=\cL M\cap\Fix(\rho)$. Thus, $M$ becomes a totally real and real analytic $n$-submanifold in $\mathcal M$. 
A biholomorphic map 
$z\mapsto f(z)$ can be complexified to yield a  biholomorphic map of $\cc^n\times \cc^n$:
\[F\colon (z,w)\mapsto\big(f(z),\bar f(w)\big).\]
Note that $F$ commutes with $\rho$ and preserves $\pi_1^{-1}(w)$ and $\pi_2^{-1}(z)$, where $\pi_1(z,w)=w$ and $\pi_2(z,w)=z$. Further, if the $f$ sends $M$ into another $n$-submanifold $\tilde M$ of the same form, then $\tilde M$ must have the same Bishop invariant of $M$ and $F$ maps $\mathcal M$ into the complexification $\widetilde {\mathcal M}$ of $\tilde M$.
It is interested to see that when $\gamma\neq0$, $\pi^{-1}(w)$ can be regarded as a discrete version of Segre varieties~\ci{We94}, which consists of only two points.
This is very different from the Segre varieties that arise from the complexification of a real analytic hypersurface~\ci{We77} which are powerful in various applications in several complex variables.  

When $\gamma\neq0$,  Moser and Webster observed that each $\pi_i\colon\mathcal M\to\cc^n$ is a two-to-one branched covering and it admits a deck transformation $\tau_i\colon\cL M\to\cL M$, and $\tau_2=\rho\tau_1\rho$. 
Being deck transformations, they satisfy
\begin{equation} \label{realization}
		z_j\circ\tau_2=z_j, \quad 
		w_j\circ\tau_1=w_j,\qquad j=1,\dots, n. 
\end{equation}
Further one can verify that  $z_1,\dots, z_n$ restricted on $\cL M$ are generators of invariant holomorphic functions of $\tau_2$,
 i.e. they are functions in $z_1,\dots, z_n$, whereas $w_1,\dots, w_n$ generate the invariant functions of $\tau_1$.

Let us identify $\mathcal M$  locally  with $\cc^n$ via local holomorphic coordinate $(z_1,w_1,x')$  with $x_\all=\frac{z_\all+w_\all}{2}$.  Thus, $\{\tau_1, \tau_2\}$ becomes a pair of involutions
on $\cc^n$.  
The linear parts of $\tau_j$ are computed as 
\begin{eqnarray}\label{eq:tau1tau2}
&\tau_1:\left(\begin{matrix} z_1 \\ w_1\\ x\end{matrix}\right)\mapsto
	\left(\begin{matrix} -z_1-\tfrac{w_1+2i\delta_nx_2}{\gamma}+\hot\\ w_1\\ x+\hot \end{matrix}\right)\!,\\	&\tau_2:\left(\begin{matrix} z_1 \\ w_1\\ x\end{matrix}\right)\mapsto
	\left(\begin{matrix} z_1 \\ -\tfrac{z_1-2i\del_nx_2}{\gamma}-w_1+ \hot\\ x+\hot \end{matrix}\right)\!.
\end{eqnarray}
%
%
Since $\tau_1,\tau_2,\rho$ are periodic mappings, the holomorphic classification of each individual map is simply  determined by its linear part. 
Their composition
$\sigma:=\tau_1\circ\tau_2$
is a germ of biholomorphism of $\cc^n$ that can be seen as a ``billiard map", which plays a central role for the classification of the triples $\{\tau_1,\tau_2,\rho\}$.

The choice of holomorphic coordinates for $\mathcal M$ is, of course, not unique. It is convenient to normalize $\rho$ first according the linear parts of $\sigma$. When $\gamma\neq0,1/2$, in suitable linear coordinates $(\xi,\eta,\zeta_2,\dots,\zeta_{n-1})$ of $\cc^n$, we have
\begin{gather}\label{tauj-1}
\tau_j(\xi,\eta,\zeta)=(\lambda_j\eta,\lambda_j^{-1}\xi,\zeta)+O(2),\quad  \la_1=\la=\la_2^{-1}; \\
\label{trhot}
\tau_2=\rho\tau_1\rho, \\
\sigma(\xi,\eta,\zeta)=(\mu\xi,\mu^{-1}\eta,\zeta)+O(2),\quad\mu=\la_1\la_2^{-1}.
\end{gather}
When $0$ is an elliptic complex tangent, we have  
\begin{equation}\label{la1>1}
\la>1,\quad \rho(\xi,\eta,\zeta)=(\bar\eta,\bar\xi,\bar\zeta).
\end{equation}
 Thus $\mu>1$ and $\sigma$ is a hyperbolic map; when $0$ is hyperbolic, we have
  \eq{|la|=1}
 |\la|=1,\quad  0<\arg\lambda\leq \pi,\quad \rho(\xi,\eta,\zeta)=(\bar\xi,\bar\eta,\bar\zeta)
  \eeq
  Thus $|\mu|=1$, and  $\sigma$ is an elliptic map. 
 When $0$ is a parabolic complex tangent, we use  coordinates $\xi=z_1+w_1$, $\eta=z_1-w_1$ and $\xi_\alpha=z_\alpha+w_\alpha$. Thus
\begin{gather}\label{tauj-2}
\tau_j(\xi,\eta,\zeta)=(\xi,\eta+(-1)^{j}2(\xi-i\delta_n\zeta_2),\zeta)+O(2),\quad \rho(\xi,\eta,\zeta)=(\bar\xi,-\ov\eta,\bar\zeta),
\\
\sigma(\xi,\eta,\zeta)=(\xi+4i\delta_n\zeta_2,\eta+4\xi-8i\delta_n\zeta_2,\zeta)+O(2).
\end{gather}
In this case, we write 
\eq{la=1}
\la_1=\la=\la_2=\mu=1. 
\eeq
Thus, the $\lambda$ is determined by
\begin{equation}\label{eq:gamma1}
	\lambda+\lambda^{-1}=\gamma^{-2}-2.
\end{equation}

The following theorem is a basic result of Moser-Webster, which transforms the classification of the real analytic $n$-submanifolds of $\cc^n$ with non-vanishing Bishop invariant to the  classification of triple involutions $\{\tau_1,\tau_2,\rho\}$ with  $\rho$ and the  linear parts specified above, and vice versa.
 
\begin{thm}[\!\!\cite{MW83}] \label{thm:MW}~
Two germs of real analytic submanifolds $M,\tilde M$ in $\cc^n$ of the form \eqref{eq:M1}-\rea{qgamma} are holomorphically (resp. formally) equivalent if and only if  there is a single holomorphic (resp. formal) coordinate change $\var$ such that   $\var^{-1}\tau_1\var=\tilde\tau_1,\var^{-1}\tau_2\var=\tilde\tau_2$ and $\var\rho\var^{-1}=\rho$, where 
$\{\tau_1,\tau_2,\rho\}$ and $\{\tilde\tau_1,\tilde\tau_2,\tilde\rho\}$ are the   Moser-Webster involutions for $M,\tilde M$ respectively. Conversely, if $\tau_1,\tau_2$ are holomorphic (resp. formal)  involutions intertwined by $\tau_2=\rho\tau_1\rho$ by an anti-holomorphic (resp. anti-holomorphic and formal) involution $\rho$, where $\tau_1,\tau_2$ have the desired linear parts given by \rea{tauj-2}, or \rea{tauj-1}-\rea{trhot} together with \rea{la1>1} or \rea{|la|=1}, then a holomorphic equivalent class of $\{\tau_1,\tau_2,\rho\}$ is realized by the holomorphic (resp. formal) equivalent class of real analytic (formal) submanifold  of the form \eqref{eq:M1}-\rea{qgamma}.
\end{thm}	
 \begin{proof}For the reader's interest, we sketch  the realization part. 
  Recall that \re{realization} says that $z_1,\dots, z_n$ are generators of invariant functions of $\tau_2$.  Given the linear parts of $\tau_1,\tau_2$ as mentioned above, the invariant functions  of $\tau_2$ in coordinates $(\xi,\eta,\zeta)$ are power series in
 \begin{gather*}
 p_1(\xi,\eta,\zeta)=\eta+\eta\circ\tau_2(\xi,\eta,\zeta),  \quad p_\alpha(\xi,\eta,\zeta)=\zeta_\alpha+\zeta_\all\circ\tau_2(\xi,\eta,\zeta),  \\ p_n(\xi,\eta,\zeta)=\eta\circ\tau_2(\xi,\eta,\zeta)\eta.
 \end{gather*}
 Since $\tau_1=\rho\tau_2\rho$, the invariant functions of $\tau_1$ in coordinates $(\xi,\eta,\zeta)$ are power series in $q_j(\xi,\eta,\zeta):=\overline{p_j(\rho(\xi,\eta,\zeta))}$ for $1\leq j\leq n$. Consider first $0<\gamma<1/2$. In this case $\rho(\xi,\eta,\zeta)=(\ov\eta,\ov\xi,\ov\zeta)$. Then define $M$ via parametrization
 $$
 M\subset\cc^n\colon z_j=p_j(\xi,\ov\xi,\zeta), \quad j=1,\dots, n,\quad \xi\in\cc, \zeta\in \rr^{n-2}.$$
It is clear that complexification $\mathcal M$ of $M$ is given by
$$
\mathcal M\subset\cc^n\times\cc^n: z_j=p_j(\xi,\eta,\zeta),\quad w_j=q_j(\xi,\eta,\zeta) \quad j=1,\dots, n,\quad (\xi,\eta,\zeta)  \in\cc^n.
$$
Since $p\circ\tau_2=p$, then $z|_{\mathcal M}$ are invariant by $\tau_2$. 

When $\gamma>1/2$, we have $\rho(\xi,\eta,\zeta)=(\ov\xi,\ov\eta,\ov\zeta)$.  Using the same construction for $\mathcal M$, define
 $$
 M\subset\cc^n\colon z_j=p_j(\xi,\eta,\zeta), \quad j=1,\dots, n,\quad (\xi,\eta,\zeta)\in \rr^{n}.$$
When $\gamma=1/2$, we have $\rho(\xi,\eta,\zeta)=(\ov\xi,-\ov\eta,\ov\zeta)$. Thus, keep equation for $\mathcal M$ and define
 $$
 M\subset\cc^n\colon z_j=p_j(\xi,\eta,\zeta), \quad j=1,\dots, n,\quad (\xi,i\eta,\zeta)\in\rr^{n}.$$
In all three cases, one can verify that   $M$ has the desired Bishop invariant $\gamma$ satisfying \re{eq:gamma1}.

 For full details, see \cite{MW83}, \cite{We92}, and \cite{AG09}.  \end{proof}
\subsection{The normal form of Moser-Webster involutions} 
Moser-Webster derived their formal normal form first for the pair of involutions $\{\tau_1,\tau_2\}$ and then the triple $\{\tau_1,\tau_2,\rho\}$ satisfying $\tau_2=\rho\tau_1\rho$.
\begin{thm}[\!\!\cite{MW83}, formal normal form of involutions] \label{mw-formal}
	Let  $\tau_1$, $\tau_2$ be a pair of formal holomorphic involutions of which the linear parts are given by \rea{tauj-1} with $\lambda\in\cc^*$.
Assume that $\mu$ is not a root of unity. There exists a unique formal {\it normalized} transformation $\psi$ commuting with $\rho$  such that
	$$\psi^{-1}\circ\tau_1\circ\psi:\left\{\begin{array}{l}
		\xi'=\Lambda(\xi\eta,\zeta) \eta\\
		\eta'=\Lambda^{-1}(\xi\eta,\zeta) \xi\\
\zeta'=\zeta
	\end{array}\right.,\quad \psi^{-1}\circ\tau_2\circ\psi:\left\{\begin{array}{l}
		\xi'=\Lambda^{-1}(\xi\eta,\zeta) \eta\\
		\eta'=\Lambda(\xi\eta,\zeta) \xi\\ \zeta'=\zeta
	\end{array}\right.,$$
	where $\Lambda$ is a power series with $\Lambda(0)=\lambda$.  Consequently if  $\mu$ is not a root of unity and
$\tau_1,\tau_2,\rho$ are given by  \rea{tauj-1}-\rea{trhot}  satisfying \rea{la1>1} or \rea{|la|=1}, then $\psi\rho=\rho\psi$, and $\Lambda(t)=\bar \Lambda(t)$ for $\la>1$  or $\Lambda(t)\cdot\bar \Lambda(t)=1$ for $|\la|=1$.
\end{thm}
Here a formal map $\psi\colon\xi'=a(\xi,\eta,\zeta),\eta'=b(\xi,\eta,\zeta), \zeta'=c(\xi,\eta,\zeta')$ is {\it normalized} if
\begin{gather*}
a(\xi,\eta,\zeta)=\xi+\sum_{j\neq k+1, j+k+|P|>1} a_{jk P}\xi^j\eta^k\zeta^P, \\ b(\xi,\eta,\zeta)=\eta+\sum_{k\neq j+1, j+k+|P|>1} b_{jk P}\xi^j\eta^k\zeta^P,\\
c_\alpha(\xi,\eta,\zeta)=\zeta_\all+\sum_{j\neq k, j+k+|P|>1} c_{\all,jkP}\xi^j\eta^k\zeta^P, \quad 1<\all<n.
\end{gather*}

\section{Elliptic CR-singularity with non-vanishing Bishop invariant}\label{sec:elli}
\setcounter{thm}{0}\setcounter{equation}{0}
The main analytic results are concerned with  $0<\gamma <\frac{1}{2}$ and are due to Moser-Webster. They first showed the following.
\begin{thm}[\!\!\cite{MW83}] When $|\la|>1$, 
 the formal power series   $\psi$ defined in \rta{mw-formal} for $\tau_1,\tau_2$  is convergent in a neighborhood of the origin. 
\end{thm}
As an application,   they proved the following.
\begin{thm}[\!\!\cite{MW83}, convergence of normalization of surfaces]
\label{NFofM}
Let $0<\gamma<\frac12$. There exists a holomorphic change of variables near the origin transforming $M$  into 
	$$M^*\colon x_n=z_1\bar z_1+\Gamma_n(x_2,\dots, x_n)(z_1^2+\bar z_1^2),\quad y_2=\cdots=y_{n-1}=0,\quad \Gamma_n(0)=\gamma.$$
Furthermore, for $n=2$, $\Gamma_2(x_2)=\gamma+\delta x_2^s$ with $\del=0$ for $s=\infty$ and $\del=0,1$ otherwise.  
\end{thm}
\begin{remark}
The normal form is useful and it
allows Moser-Webster~\cite{MW83} to show that 	near an  elliptic tangent with $\gamma\neq0$, there is an analytic family of holomorphic discs defined by $z_j=c_j\in\rr$ for $j>1$  and  $|z_1|^2+\Gamma_n(c_2,\dots, c_n)(z_1^2+\bar z_1^2)\leq c_n$ whose boundaries are ellipse contained in the normalized submanifold $M^*$,  and the union of these discs is the local hull of holomorphy of $M^*$.
\end{remark}

\section{Elliptic CR-singularity surface with vanishing Bishop invariant}\label{gamma=0}
When $\gamma=0$, the method of Moser-Webster involutions $\{\tau_1,\tau_2,\rho\}$ breaks down. For instance, for the complexification of $w=z\ov z$, the $\pi_1,\pi_2$ are no longer branched coverings as $\pi_1^{-1}(0)$ is a line. 
When $Q_0$ is perturbed, for instance let $M\subset\cc^2$ be defined by   $w=z\ov z+z^3+\bar z^3$. Then the complexification of $M$ yields two $3$-to-1 branched coverings $\pi_i\colon\cL M\to \cc^2$; one can verify that the only deck transformation of each $\pi_i$ is the identity map. 

In \cite{moser-zero} (see also \cite{HK95, huang-surv}), Moser carried out a study for $\gamma=0$ from a  formal power series point of view. When $M$ is not formally equivalent to $Q_0$, 
Moser derived the following  formal pseudo-normal form for $M$ with
$\gamma=0$:
\begin{equation}
	w=z\bar{z}+z^s+\bar{z}^s+2\RE\{\sum_{j\ge s+1}a_{j}z^j\}.
\end{equation}
Here $s$ is the simplest higher order invariant of $M$ when $\gamma=0$, which we call it the {\it Moser
invariant}. Moser showed the following convergence result.
\begin{thm}[\!\!\cite{moser-zero}]
If a real analytic surface $M$ is formally equivalent to the quadric $Q_0=\{(z,w)\in
{\mathbb C}^2: w=|z|^2\}$, then $M$ is
holomorphically equivalent to  $Q_0$.
\end{thm}

Note that there are formal biholomorphisms of $Q_0$ that are divergent, which leads to substantial difficulty to further classify the Moser pre-normal form. 
In a complete new approach,
Huang-Yin~\ci{HY09} developed a method for the holomorphic classification for surfaces with vanishing Bishop invariant. Their formal classification is the following
\begin{thm}[\!\!\cite{HY09}]
	Let $M$ be a formal  Bishop surface with  vanishing Bishop invariant at $0$  and a finite
	Moser invariant $s$. Then
	there exists a formal transformation
	$(z',w')=F(z,w)$ with $ F(0,0)=(0,0)$
	such that
	in the $(z',w')$
	coordinates, $M'=F(M)$ is given  by  
	\begin{equation}\label{HYnf}
		w'=|z'|^2+z'^s+\bar{z'}^s+\varphi(z')+\overline{\varphi(z')}
	\end{equation}
	where
	$
	\varphi(z')=\sum\limits_{k=1}^{\infty}\sum\limits_{j=2}^{s-1}a_{ks+j}(z')^{ks+j}.
	$
Furthermore,  $\tilde F$ is another such formal transformation for a possibly different $\var$ in the same form, if and only if $\tilde F=R\circ F$ with $R(z',w')=(\mu z',w')$  and $\mu^s=1$.
\end{thm}
The following results are fundamental and provide a holomorphic classification of the surfaces with $\gamma=0$.
\begin{thm}[\!\!\cite{HY09}]
Let $M$ and $M'$ be  real analytic Bishop surfaces with vanishing Bishop invariant and finite  Moser invariant.
	Suppose that
	$F$ is a formal biholomorphic map that transforms $M$ into $M'$.   Then $ F$ is convergent.
\end{thm}
 
An important step in their proof is  for the convergence of $F$  when $F$ is reduced to the case that
it is  
 tangent to the identity map via Lemma~2.1 in~\cite{HY09}.
\begin{cor}[\!\!\cite{HY09}]
Let $M_1$ and $M_2$ be  real analytic Bishop surfaces with $\gamma=0$ and $s\not =\infty$ at $0$. Suppose that $M_1$ has
	a formal normal form:
	$$w'=z'\bar{z'}+{z'}^s+\bar{z'}^s+2\RE\{\sum\limits_{k=1}^{\infty}\sum\limits_{j=2}^{s-1}
	a_{ks+j}({z'})^{ks+j}\};
	$$
	and suppose that $M_2$ has a formal normal form:
	$$w'=z'\bar{z'}+{z'}^s+\bar{z'}^s+2\RE\{\sum\limits_{k=1}^{\infty}\sum\limits_{j=2}^{s-1}
	b_{ks+j}({z'})^{ks+j}\}.$$
	Then $(M_1,0)$ is biholomorphic to $(M_2,0)$ if and only if there is a constant $\theta$, with $e^{s\theta i}=1$,
	such that $a_{ks+j}=e^{\theta j i}b_{ks+j}$ for any
	$k\ge 1$ and $j=2,\cdots,s-1$.
\end{cor}

One can also consider normal forms of real analytic submanifolds under volume-preserving homomorphisms. Motivated by Moser's results, the following is proved.
\begin{thm}[\!\!\cite{Go94}]
Let $M\subset\cc^n$ be a real analytic surface with vanishing Bishop invariant. Then there exists a unique formal biholomorphism $\psi$ that preserves $dz\wedge dw$ such that $\psi(M)$ is given by
\begin{equation}\label{vol-nf}
	w=z\bar{z}+2\RE\{z^2\sum_{j+k\geq1, j\geq k}z^{j} \bar z^k\}.
\end{equation}
Furthermore, the $\psi$ converges if all $a_{jk}$ vanish.
\end{thm}

\section{Non-exceptional hyperbolic surfaces 
}\label{gamma>.5}
\setcounter{thm}{0}\setcounter{equation}{0}
When $M$ has an {\it non-exceptional} hyperbolic complex tangent, i.e. the Moser-Webster invariant $\mu$ is not a root of unity and $|\mu|=1$, the formal normal form of Moser-Webster \rt{NFofM} still applies (with $\gamma>1/2$ of course). Note that the normal form is contained in the hyperplane $\IM z_n=0$.  This turns out to be one of obstructions for the convergence of the normalization. Moser-Webster proved the following result.
\begin{prop}[\!\!\cite{MW83}]\label{non-flat-hyper}
	If $1/2<\gamma<\infty$ then the hyperbolic surface $z_1=z_1\bar z_1+\gamma \bar z_1^2+\gamma z_1^3\bar z_1$ cannot be transformed into a real hyperplane by any (convergent) biholomorphic transformation.
\end{prop}
When $\gamma$ is exceptional, the proof is much simpler by showing that $M$ cannot be flattened to sufficient higher order. For non-exceptional $\gamma$, the proof uses Birkhoff's theorem on the existence of elliptic area-preserving mappings.
\subsection{Non-flattenable hyperbolic surfaces}
We have seen that for the Moser-Webster involutions $\{\tau_1,\tau_2,\rho\}$, the invariant functions of $\tau_1,\tau_2$ play an important role it realizing $\{\tau_1,\tau_2,\rho\}$ for a real analytic manifold $M$. It was observed by Moser-Webster that the existence of non-trivial holomorphic functions that are both invariant by $\tau_1,\tau_2$ is equivalent to the flatness of $M$, i.e. there exists a local biholomorphic map $\psi$ such that $\psi(M)$ is contained in $\IM z_2=0$. The following is in Moser-Webster.
\begin{prop}\label{firt-int}
Let $\{\tau_1,\tau_2,\rho\}$ be the involutions associated to $M\subset\cc^2$ with $\gamma\neq0$. Then  there exists a holomorphic function $h$ on $\mathcal M$ such that $h\circ\tau_j=h$ for $j=1,2$ and its Taylor polynomial of degree $2$ is not constant, if and only if $M$ can be holomorphically flattened. 
\end{prop}
\begin{proof}We may assume that $h(0)=0$ and $M$ has the form $z_2=F_2(z_1,\bar z_1)=q_\gamma(z_1,\ov z_1)+O(3)$. We use $(z_1,w_1)$ has holomorphic coordinates for $\mathcal M$. Then $h(z_1,w_1)= h_2(z_1,z_2)$ and $h(z_1,w_1)=h_1(w_1,w_2)$. By a straightforward computation, we get $$
h_2(z_1,q_{\gamma}(z_1,w_1))=cq_{\gamma}(z_1,w_1)+O(3)
$$
with $c\neq0$.  Recall that $M$ is identified with real and real analytic surface $M_0$ in $\mathcal M$ via embedding $z_1\to (z_1,\ov z_1)$.  Then $\tilde h =\RE(c^{-1} h )$ is also invariant by $\tau_1,\tau_2$. Let $g(z_1,w_1)$ be the unique holomorphic function on $\mathcal M$ that is invariant under $\tau_2$ such that $g=\tilde h$ on $M_0$. Then $g(z_1,w_1)=\tilde g(z_1,F_2(z_1,w_1))$. Thus  $z\to(z_1,\tilde g(z))$ transforms $M$ into $\cc\times\rr$, since $\tilde g(z_1,F_2(z_1,\bar z_1))$ is real valued.
\end{proof}

\subsection{Flatten hyperbolic surfaces}  
For the flatten case, the divergence of normalization still persists.
\begin{thm}[\!\!\cite{Go96a}]
For each non-exceptional $\gamma\in(1/2,\infty)$ there exists a real analytic surface $M\subset\cc\times\rr$ that cannot be transformed into the Moser-Webster normal form by any holomorphic change of coordinates.
\end{thm}
 Inspired by a theorem of Siegel~\ci{Si54} on Hamiltonian systems, a more precise theorem \cite[Thm. 6.1]{Go96a} shows that in suitable space of flatten real analytic surfaces its corresponding $\sigma$ of a generic surface  any prescribed non-exceptional hyperbolic complex tangent has curves of periodic points in the complexified surface $\cL M$ accumulating at the origin, but the curves intersect the total real surface $M$ in $\cL M$ at isolated points only (if the intersection is non-empty). Consequently, the normal form of $\sigma$ cannot be achieved by convergent transformations, although $\sigma$ admits a non-constant invariant function by \rp{firt-int}. This is in contrast with integrable Hamiltonian systems: a Hamiltonian function in $\rr^2$ with non-resonate eigenvalue and a non-constant first-integral can always be transformed into the Birkhoff normal form.

 \subsection{Surfaces   equivalent to $Q_\gamma$}
Motivated by the results of Moser-Webster and using their involutions, Gong proved.

\begin{thm}[\!\!\cite{Go94a}]
Let $M$ be formally equivalent to $Q_\gamma$ with $1/2<\gamma<\infty$. 
	Assume that the invariant   $\lambda$ of $M$    satisfies the Diophantine condition, i.e. 
		$$ |\lambda^n-1|>\frac{c}{n^\delta},\quad n=1,2\dots, \quad c>0.$$
	Then,  $M$ is holomorphically equivalent to $Q_\gamma$.  
\end{thm}
It turns out that for the above result,  a small divisor condition is necessary. Indeed, one has the following.
\begin{thm}[\!\!\cite{Gon04}]
There exists a real analytic surface $M$ with non-exceptional hyperbolic Bishop invariant $\gamma$ such that $M$ is formally but not holomorphically equivalent to $Q_\gamma$.
\end{thm}

\subsection{A KAM theory for  general  hyperbolic surfaces}
It is evident that the Moser-Webster theory is already rich for the surfaces $M$ in $\cc^2$. Important dynamics phenomena already occur in the surface case. We shall now focus on a recent  development  for the surface case. 
We return to
  $$\tau_1:\left\{\begin{array}{l}
	\xi'=\lambda\eta+{\rm h.o.t.} \\
	\eta'=\lambda^{-1}\xi+{\rm h.o.t.}
\end{array},\right.\ \tau_2:\left\{\begin{array}{l}
	\xi'=\lambda^{-1}\eta+{\rm h.o.t.} \\
	\eta'=\lambda\xi+{\rm h.o.t.}
\end{array},
\right. $$
$$\lambda:=e^{\frac{1}{2}i\alpha},  \tfrac{\alpha}{\pi}\in\rr\setminus\qq, \quad \Lambda(\xi\eta)=\lambda+\sum_{n\geq 1}\tilde c_n (\xi\eta)^n.$$
Here $\Lambda$ appears in the Moser-Webster formal normal form of $\{\tau_1,\tau_2\}$.
The following theorem is a holomorphic KAM-like result that shows the holomorphic dynamic has a "lot" of invariant analytic sets in a neighborhood of its fixed point. These invariant sets are biholomorphic  to $\{\xi\eta=\omega\}$ intersecting  a neighborhood $\{|\xi|,|\eta|<r\}$ of the origin.  This phenomenon takes its root in a real Hamiltonian dynamical system on $\Tt^n\times\rr^{n}$ which are small perturbations of completely integrable system of the form $\dot \theta= \omega(I),\; \dot I= 0$, $(\theta,I)\in\Tt^n\times\rr^{n}$, $\omega$ analytic. For the later, any torus $\Tt^n\times\{I_0\}$ is invariant and the motion on it is a rotation of angle $\omega(I_0)$. KAM theory \cite{bhs-book}, named after Kolmogorov-Arnold-Moser \cite{Kolmogorov,arnold-kol, moser-anneau}, asserts that under some non-degeneracy conditions, a small enough analytic perturbation still has a lot of invariant manifolds diffeomorphic to a torus on which the dynamics is conjugate to a rotation. A similar phenomena was found for reversible mappings on  $\Tt^n\times\rr^{m}$ by Sevryuk\cite{sevryuk-lnm}. More recently, in the framework of local holomorphic vector fields near a fixed point, Stolovitch shown the existence of a lot of analytic invariant sets biholomorphic to {\it resonant manifolds} on which the dynamic is holomorphically linearizable \cite{Stolo-kam}. The ideas developed there served as stepping stone for the following result~:
\begin{thm}[\!\!\cite{stolo-zhao-cr}] Assume $\Lambda(\xi\eta)\neq \lambda$.
	If $r>0$ is small enough, there exists a "asymptotic full measure" parameters set ${\mathcal O}_{r}\subset]-r^2,r^2[$  such that
	$\forall \  \omega\in{\mathcal O}_{r}$, $\exists$ $\mu_\omega\in\rr$ and an holomorphic transformation $\Psi_{\omega}$, Whitney smooth in $\omega$, on  
	${\mathcal C}_{\omega}^{r}:=\{\xi\eta=\omega, \ |\xi|,|\eta|<r\}$ with $\Psi_{\omega}\circ\rho=\rho\circ\Psi_{\omega}$
	and such that  on  ${\mathcal C}_{\omega}^{r}$,
	$$\Psi_{\omega}^{-1}\circ \tau_1\circ\Psi_{\omega}:\left\{\begin{array}{l}
		\xi'=e^{\frac{\rm i}{2}\mu_{\omega}} \eta\\
		\eta'=e^{-\frac{\rm i}{2}\mu_{\omega}} \xi
	\end{array}\right.,\;\
	\Psi_{\omega}^{-1}\circ
	\tau_2\circ\Psi_{\omega}:\left\{\begin{array}{l}
		\xi'=e^{-\frac{\rm i}{2}\mu_{\omega}}\eta\\
		\eta'=e^{\frac{\rm i}{2}\mu_{\omega}}\xi
	\end{array}\right..$$
	Here, "asymptotic full measure" means $\frac{| {\mathcal O}_{r}|}{2r^2}\stackrel{r\to 0}{\longrightarrow} 1$.
\end{thm}

In other words, $\Psi_{\omega}({\mathcal C}_{\omega}^{r})$ is a holomorphic invariant set of the $\tau_i$'s and their restrictions  are
 conjugated to  linear maps by a single transformation $\Psi_\omega$.
\begin{figure}[hbtp]
	\begin{center}
		\leavevmode
		\includegraphics[scale=.35]{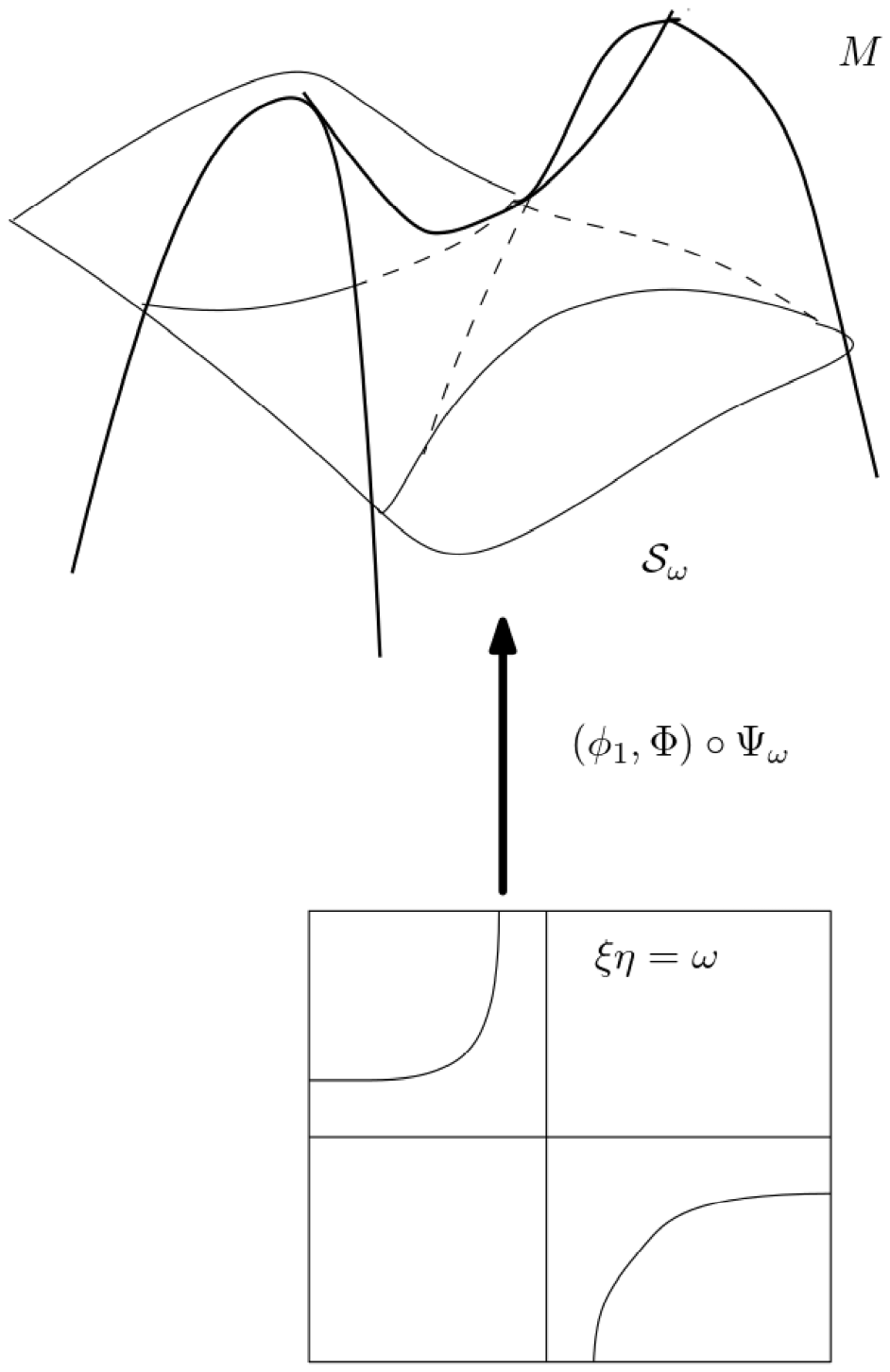}
		\caption{Holomorphic hyperbola~: Intersection of $M$ by a holomorphic curve.}\label{figure1}
	\end{center}
\end{figure}
From this dynamical result, we can infer some information on the real surface near its CR singularity~:
\begin{thm}[\!\!\cite{stolo-zhao-cr}]
	Let $M$ be a surface with an hyperbolic CR singularity at the origin which is non exceptional and not formally equivalent to a quadric. Then, there exist a neighborhood of the origin and a Whitney smooth family of holomorphic curves $\{{\mathcal S}_{\omega}\}_{\omega\in{\mathcal O}} $ which intersects $M$ along holomorphic hyperbolas. The latter are collections of 2 real curves which are simultaneously holomorphically mapped to the two branches of the same hyperbolas $\xi\eta=\omega$, $\omega\neq0$ (See Figure \ref{figure1}).
\end{thm}

\section{Exceptional hyperbolic surfaces 
}\label{sect:exec}
\setcounter{thm}{0}\setcounter{equation}{0}

\subsection{Reversible parabolic diffeomorphisms}

Let $(\phi,\tau)$ be a pair of a reversible map $\phi$ and its reversing involution $\tau$
\begin{equation}\label{eq:reversible}
	\tau^{\circ 2}=\id,\qquad  \tau\circ\phi\circ\tau=\phi^{\circ(-1)}.
\end{equation}
Denote by $\cL G$ the group of diffeomorphisms generated by $\{\phi,\tau\}$. Then
\[\cL G=\big\{\phi^{\circ n}\mid n\in\zz\big\}\cup\big\{\tau\circ\phi^{\circ n}\mid n\in\zz\big\},\]
where each $\tau_{n+1}=\tau\circ\phi^{\circ n}$ is an involution reversing $\phi$.
For every $n\in\zz$ the pair of involutions $(\tau_n,\tau_{n+1})$ satisfies $\tau_n\circ\tau_{n+1}=\phi$ and therefore generates $\cL G$.
\emph{Thus the problem of classification of reversible maps $(\phi,\tau)$ with respect to conjugation is equivalent to that of pairs of involutions $(\tau_n,\tau_{n+1})$}, i.e. if $\tilde\tau_{n+1}=\tilde\tau\circ \tilde\phi^{\circ n}$ with $\tilde\phi^{-1}=\tau\circ \tilde \phi\circ \tau$ and $\tilde\tau^2=\id$, then  $(\tau_k,\tau_{k+1})$ is equivalent to $(\tilde\tau_\ell,\tilde\tau_{\ell+1})$ if and only if  $k=\ell$ and  there is a diffeomorphism $\psi$ such that $\psi^{-1}\circ \tau_j\circ \psi=\tilde\tau_j$ for all $j=k,k+1$, and the latter is equivalent to $\psi^{-1}\circ \phi\circ \psi=\tilde\phi$ and $\psi^{-1}\circ \tau\circ \psi=\tilde\tau$ 
as $\psi^{-1}\circ \phi\circ \psi=\psi^{-1}\circ \tau_j\circ\tau_{j+1}\circ \psi=\tilde\phi$ and   $\psi^{-1}\circ \tau\circ \psi=\psi^{-1}\circ \tau_{j+1}\phi^{-j}\circ \psi=\tilde\tau$.
Thus,
one may consider just the pair
\[(\tau_1,\tau_2)=(\tau,\tau\circ\phi).\]

\goodbreak

\begin{assumption}\label{assumptions}
	We  assume that $(\phi,\tau)$ are holomorphic diffeomorphisms of $(\cc^2,0)$, such that $\phi\neq\tau$ and they satisfy three conditions:
	\begin{enumerate}
		\item $\phi\in\Diff(\cc^2,0)$ is \emph{parabolic}: $\phi^{\circ p}=\id+\hot$ for some positive integer $p\geq 1$ (the minimal with such property),
		\item $\tau\in\Diff(\cc^2,0)$ is a \emph{holomorphic reflection}: an involution whose linear part has eigenvalues  $\{1,-1\}$, which reverses $\phi$,
		\[	\tau^{\circ 2}=\id,\qquad  \tau\circ\phi\circ\tau=\phi^{\circ(-1)},\]
		\item the pair $(\phi,\tau)$ possesses an \emph{analytic first integral of Morse type}, i.e. with non-degenerate critical point at $0$, $H(0)=0$, $DH(0)=0$, $\det D^2H(0)\neq 0$,
		\[H=H\circ\phi=H\circ\tau.\]
	\end{enumerate}	
\end{assumption}

Up to a linear change of variables, 
they take the form
\begin{equation}\label{eq:phitau}
	\phi(\xi)=\Lambda\xi+\hot(\xi),\qquad 	\tau(\xi)=\sigma\xi+\hot(\xi),	\qquad H(\xi)=\xi_1\xi_2+\hot(\xi),
\end{equation}	
where
\begin{equation}\label{eq:sigmaLambda}
	\sigma=\left(\begin{smallmatrix}0\, &1\\[4pt]1&\, 0\end{smallmatrix}\right),
	\qquad 	 \Lambda=\begin{cases}
		\left(\begin{smallmatrix}\lambda & 0 \\[4pt] 0 & \lambda^{-1}\end{smallmatrix}\right),& \Lambda^{p}=I,\quad p\geq 1,\\[6pt]
		-\sigma,& \Lambda^{2}=I,\quad p=2.
	\end{cases}
\end{equation}
The case of diagonal $ \Lambda=	\left(\begin{smallmatrix}\lambda & 0 \\[3pt] 0 & \lambda^{-1}\end{smallmatrix}\right)$ arises when the second involution
$\tau_2=\tau\circ\phi$ is a holomorphic reflection as well, while
the case $\Lambda=-\sigma$ happens when $\tau_2=\tau\circ\phi$ is tangent to $-\id$.
We shall note that a possibility of $\tau_2=\tau\circ\phi$ being tangent to $\id$ is excluded by the assumption that $\phi\neq\tau$, since any involution tangent to the identity is in fact the identity.

The diffeomorphism $\phi^{\circ p}(\xi)=\xi+\hot(\xi)$ is tangent to the identity, and as such it possesses  a unique formal infinitesimal generator (see e.g. \cite[Thm. 3.17]{ilyaskenko-yakovenko})~: a formal vector field $\hat\bX(\xi)$ whose formal time-1-flow $\exp(\hat\bX)(\xi)$ is equal to the Taylor expansion of $\phi^{\circ p}(\xi)$.  This formal vector field $\hat\bX(\xi)$ has $H(\xi)$ as a first integral, and is reversed by $\tau$:  $\tau^*\hat\bX=-\hat\bX$.
This allows to reduce the problem of formal classification of $(\phi,\tau)$ to a formal classification of such formal vector fields $\hat\bX$.

\subsection{Formal classification of reversible parabolic diffeomorphisms}

\begin{thm}[\!\!\cite{stolo-klimes}, formal classification]\label{thm:1} 
 Let $(\phi,\tau)$ be a pair of a reversible map $\phi$ and its reversing involution $\tau$ as in \re{eq:reversible} with a first integral $H$ be as above satisfying Assumptions~\ref{assumptions}.
	
	There exists a formal transformation $\xi\mapsto\hat\Psi(\xi)\in\widehat{\Diff}_{\id}(\cc^2,0)$ and a formal diffeomorphism $\hat G\in\hatDiff_{\id}(\cc,0)$,
	such that
	\[\hat\Psi\circ\phi=\hat\phi_{\rm nf}\circ\hat\Psi, \qquad  \hat\Psi\circ\tau=\sigma\,\hat\Psi, \qquad \hat G(H)=h\circ\hat\Psi, \]
	where
	\begin{equation*}
		\hat\phi_{\rm nf}(\xi)=\Lambda\cdot\exp(\tfrac{1}{p}\hat{\bX}_{\rm nf})(\xi),\quad
		\text{with $\sigma,\Lambda$ as in \eqref{eq:sigmaLambda}, \quad and \quad $h=\xi_1\xi_2.$}
	\end{equation*}
	Here  $\hat{\bX}_{\rm nf}(\xi)=\Lambda^*\hat{\bX}_{\rm nf}(\xi)=-\sigma^*\hat{\bf X}_{\rm nf}(\xi)$ is one of the following vector fields:
	\[\hat{\bX}_{\rm nf}(\xi)=\begin{cases}
		\textnormal{(o) }\ 0,&\\[6pt]
		\begin{aligned}
			&	\textnormal{(a) }\ c\,h^s\big(\xi_1\tdd{\xi_1}-\xi_2\tdd{\xi_2}\big), \\[3pt]
			&	\textnormal{(b) }\ c\,h^s\frac{P(u,h)}{1+c\,\hat\mu(h)P(u,h)}\big(\xi_1\tdd{\xi_1}-\xi_2\tdd{\xi_2}\big), 	
		\end{aligned}\ \ \Bigg\} &
		\text{if}\ \Lambda=\left(\begin{smallmatrix}\lambda & 0\\[3pt]0&\lambda^{-1}\end{smallmatrix}\right)\!, \ p\geq 1, \\[24pt]
		\textnormal{(c) }\ c\,h^s P(u, h)\big(\xi_1\tdd{\xi_1}-\xi_2\tdd{\xi_2}\big),
		& \text{if}\ \Lambda=-\sigma, \ p=2,
	\end{cases}\]
	where $c\neq0$, and
	\begin{itemize}
		\item[\textnormal{(b)}] $P(u,h)$ is polynomial in $u(\xi):=\xi_1^p+\xi_2^p$ of order $k\geq 1$,
		\[P(u,h)=u^k+ P_{k-1}(h)u^{k-1}+\ldots+P_0(h),\quad P(u,0)=u^k,\]
		and $\hat\mu(h)=\sum_{n=0}^{+\infty}\mu_nh^n$ is a formal power series,
		\item[\textnormal{(c)}]
		$P(u, h)$ is an odd polynomial in $u=\xi_1+\xi_2$ of order $2\tilde k+1\geq 0$,
		\[P(u, h)=u^{2\tilde k+1}+ P_{2\tilde k-1}(h)u^{2\tilde k-1}+\ldots+ P_1(h)u,\quad P(u, 0)=u^{2\tilde k+1}.\]
	\end{itemize}
\end{thm}	
\begin{rem}

	\begin{itemize}
		
		\item The case \textnormal{(o)} happens if and only if $\phi^{\circ p}=\id$, and there exists a normalizing transformation $\hat\Psi$ which is convergent.
		
		\item In  cases \textnormal{(a),\,(b),\,(c)} the formal normalizing transformation $\hat\Psi\in\widehat{\Diff}_{\id}(\cc^2,0)$ is unique.
		Furthermore, in the cases \textnormal{(b),\,(c)} $h\circ\hat\Psi$ is convergent.
		
		\item The formal equivalence class of $(\phi,\tau)$ with respect to conjugation by the group $\widehat{\Diff}_{\id}(\cc^2,0)$ contains a unique representative in the above formal normal form $(\hat\phi_{\rm nf},\hat\tau_{\rm nf})$.
		
		\item In the formal equivalence class of $(\phi,\tau)$ with respect to conjugation in the full group $\widehat{\Diff}(\cc^2,0)$ the above formal normal form $(\hat\phi_{\rm nf},\hat\tau_{\rm nf})$ and its infinitesimal generator $\hat{\bf X}_{\rm nf}$ are determined uniquely up to the action of scalar transformation
		$\xi\mapsto\zeta\cdot\xi$, $\zeta\in\cc^*$, and also of $\xi\mapsto\sigma\xi$ in case $p\in\{1,2\}$ when $\sigma\Lambda=\Lambda\sigma$, by which the constant $c\neq 0$ can be further normalized.
		
		\item The group $\cL Z(\hat\phi_{\rm nf},\sigma)$ of formal diffeomorphisms commuting with $(\hat\phi_{\rm nf},\sigma)$ is the same as the group
		$\cL Z^{\sigma,\Lambda}(\hat{\bf X}_{\rm nf})$ of formal $(\sigma,\Lambda)$-equivariant diffeomorphisms preserving $\hat{\bf X}_{\rm nf}$,
		$\cL Z(\hat\phi_{\rm nf},\sigma)=\cL Z^{\sigma,\Lambda}(\hat{\bf X}_{\rm nf})$.
		In the cases \textnormal{(a),\,(b),\,(c)} it is identified with a subgroup of $\zz_{2kp+4s}$ acting on $(\cc^2,0)$ by $\xi\mapsto e^{\frac{\pi i r}{kp+2s}}\sigma^r\xi,\ r\in\zz_{2kp+4s}$, where
		\[0<kp+2s:=\begin{cases}
			\textnormal{(a)}\ 2s,\\
			\textnormal{(b)}\ kp+2s,\\
			\textnormal{(c)}\ 2\tilde k+1+2s.
		\end{cases}\]
		is the multiplicity of the fixed point divisor $\Fix(\phi^{\circ p})$ at the origin.

	\end{itemize}
\end{rem}

\begin{remark}
		The variables $h=\xi_1\xi_2$ and $u=\xi_1^p+\xi_2^p$, resp. $\tilde u=(\xi_1+\xi_2)^2$, are basic $(\sigma,\Lambda)$-invariant functions: any formal/analytic $(\sigma,\Lambda)$-invariant function can be written as a formal/analytic function of $(h,u)$, resp. $(h,\tilde u)$.
		
\end{remark}

In the case (a)	of Theorem~\ref{thm:1}, there exist analytic germs that are formally but not analytically equivalent to the normal form. 
In fact, there are topological obstructions to convergence.

The cases (b) and (c) of Theorem~\ref{thm:1} carry close analogy with the Birkhoff--\'Ecalle--Voronin theory of parabolic diffeomorphisms in dimension 1.
While the formal normalizing transformation is generically divergent, the obstructions to convergence are of a purely analytic nature and can be expressed in terms of
an infinite-dimensional functional modulus  (Theorem~\ref{thm:analytic} below).


\begin{remark}
	The formal classification of Theorem~\ref{thm:1} is quite similar to the study of 1-parameter families of holomorphic germs $\phi_{\epsilon}(z)$
	unfolding a parabolic germ $\phi_0(z)=\lambda z+\hot(z)$, $\lambda^{p}=1$.
	The study of such families in the finite codimension case $s=0$ was carried independently by C.~Christopher, P.~Marde\v{s}i\'c, R.~Roussarie \& C.~Rousseau \cite{
		Rousseau-Christopher, Rousseau10, 
		Rousseau} and by J.~Ribon \cite{Ribon-f, Ribon-a, Ribon-c}.\footnote{Prior to that, this was investigated also by J.~Martinet \cite{Martinet}, P.~Lavaurs \cite{Lavaurs}, 
		and A.~Glutsyuk \cite{Glutsyuk}.}
	It leads to a modulus of analytic classification that ``unfolds'' the Birkhoff--\'Ecalle--Voronin modulus.
	
\end{remark}

Since the formal normal form $(\hat\phi_{\rm nf},\sigma)$ of Theorem~\ref{thm:1} in the case (b) is a priori purely formal (due to the formal invariant $\hat\mu(h)$), we introduce instead a larger \emph{model class} represented by an analytic model.

\begin{def}\label{def:model}
	Let us introduce a \emph{model} $(\phi_{\rm mod},\sigma)$ for $(\phi,\tau)$, as $\phi_{\rm mod}=\Lambda\exp\big(\tfrac{1}{p}{\bf X}_{\rm mod} \big)$, where
	\begin{equation}\label{eq:Xmodel}
		{\bf X}_{\rm mod}(\xi)=\begin{cases}
			\text{(0) }\ 0,\\
			\text{(a) }\ c\,h^s\big(\xi_1\tdd{\xi_1}-\xi_2\tdd{\xi_2}\big), \\[6pt]
			\text{(b) }\ c\,h^sP(u,h)\big(\xi_1\tdd{\xi_1}-\xi_2\tdd{\xi_2}\big), \\[6pt]
			\text{(c) }\ c\,h^s P(u, h)\big(\xi_1\tdd{\xi_1}-\xi_2\tdd{\xi_2}\big)
		\end{cases}
	\end{equation}
	with the same $c$ and $P(u,h)$ as in the formal normal form.
	The {\it model class} of $\phi_{\rm mod}$ is the set of all analytic $\phi$ with the same model, i.e. it is the union of formal classes with over all invariants $\hat{\mu}(h)$.
\end{def}

When $(\phi,\tau)$ is of formal type \textnormal{(b)} or \textnormal{(c)}, then
we can show that up to an analytic tangent-to-identity change of coordinates one can assume that $\tau(\xi)=\sigma\xi$,  $h\circ\phi(\xi)=h(\xi)$ and
\begin{equation}\label{eq:prenormalphi}
	\xi_j\circ\phi(\xi)=\xi_j\circ\phi_{\rm mod}(\xi)\mod h^{s}P(u,h)^2\xi_j, \quad j=1,2,
\end{equation}
is in so called \emph{prenormal form}.

\subsection{Classification on sectorial domains}

\begin{defn}\label{def:sector}
	A \emph{cuspidal sector} with vertex at $h=0$ is a simply connected planar domain $S$ bounded by two real analytic curves, each of which has an asymptotic tangent at the vertex, and by an arc of a fixed radius (see Figure~\ref{figure:sectors}).
	The \emph{angular opening} of the cuspidal sector is the angle between the tangent rays of the two bounding curves at the vertex.	
	
	\emph{We shall consider the vertex as included in the cuspidal sector, $0\in S$.} That way the limit situation at $h=0$ will be part of our description.
\end{defn}

\emph{Let us stress that a covering of a neighborhood of 0 by a collection of cuspidal sectors of positive angular openings may not contain any finite sub-covering.}

The following is our analogy of Theorem~\ref{thm:kimura}.

\begin{thm}[\!\!\cite{stolo-klimes},``sectorial'' equivalence]\label{thm:sectorial}
	Let $(\phi,\sigma)$ of formal type \textnormal{(b)} or \textnormal{(c)} be in the prenormal form \eqref{eq:prenormalphi},
	and let $(\phi_{\rm mod},\sigma)$ be its model.
	There exists a countable collection of cuspidal sectors\footnote{See Figure~\ref{figure:sectors} and Definition~\ref{def:sector}.}
	covering a disc $\{|h(\xi)|<\delta_2\}$ for some $\delta_2>0$, and for each given sector $S$ there exists a covering,  by a $(\sigma,\Lambda)$-invariant family\footnote{If $\Omega_{S}$ is a domain in the family, then its images $\sigma(\Omega_S)$ and $\Lambda(\Omega_S)$ are also in the family.} 
	of $4kp$ ``Lavaurs domains'' $\{\Omega_{S}^j\}_{j=1,\ldots, 4kp}$,
   of the set $B_S\setminus\Fix(\phi_{\rm mod}^{\circ p})$ with
	\begin{equation}\label{eq:BS}
		B_S=\{\xi\in\cc^2: |\xi|<\delta_1,\ h(\xi)\in S\}
	\end{equation}
	(see Figure~\ref{figure:sectors}), and a family of bounded analytic transformations $\{\Psi_{\Omega_{S}^j}\}_{j=1,\ldots, 4kp}$,
	such that
	\[\Psi_{\Lambda(\Omega_{S}^j)}\circ\phi=\phi_{\rm mod}\circ\Psi_{\Omega_{S}^j},\qquad \Psi_{\sigma(\Omega_{S}^j)}\circ\sigma=\sigma\Psi_{\Omega_{S}^j},\qquad h\circ\Psi_{\Omega_{S}^j}=h,\]
	for all $j$ and $S$.
	We call the family $\{\Psi_{\Omega_{S}^j}\}_{j=1,\ldots, 4kp}$ a \emph{normalizing cochain}.
	Such normalizing cochain is unique up to left composition with cochains of flow maps
	\begin{equation}\label{eq:cochainofflowmaps}
		\Big\{\exp\big(h^{-s}C_{\Omega^{j}_{S}}(h){\bf X}_{\rm mod}\big)\Big\}_{j=1,\ldots 4kp},
	\end{equation}
	where the $C_{\Omega^{j}_{S}}(h)$ are bounded analytic functions on $S$.
\end{thm}

The form of the domains $\Omega_S^j$ in the covering of Theorem~\ref{thm:sectorial} is determined by the dynamics of the model vector field
${\bf X}_{\rm mod}$ \eqref{eq:Xmodel}.
The set $B_S$ \eqref{eq:BS} has two ``essential'' boundary components: ``outer'' one at $\{|\xi_1|=\delta_1\}$ and ``inner'' one $\{|\xi_2|=\delta_1\}$, and the $4kp$ domains $\Omega_{S}^j$ are correspondingly grouped into two sets: $2kp$ cyclically ordered \emph{outer domains} (touching the ``outer'' boundary) and $2kp$ cyclically ordered \emph{inner domains} (touching the ``inner'' boundary), see Figure~\ref{figure:sectors}.

\begin{figure}[h!t]
	\centering
	\includegraphics[width=0.99\textwidth]{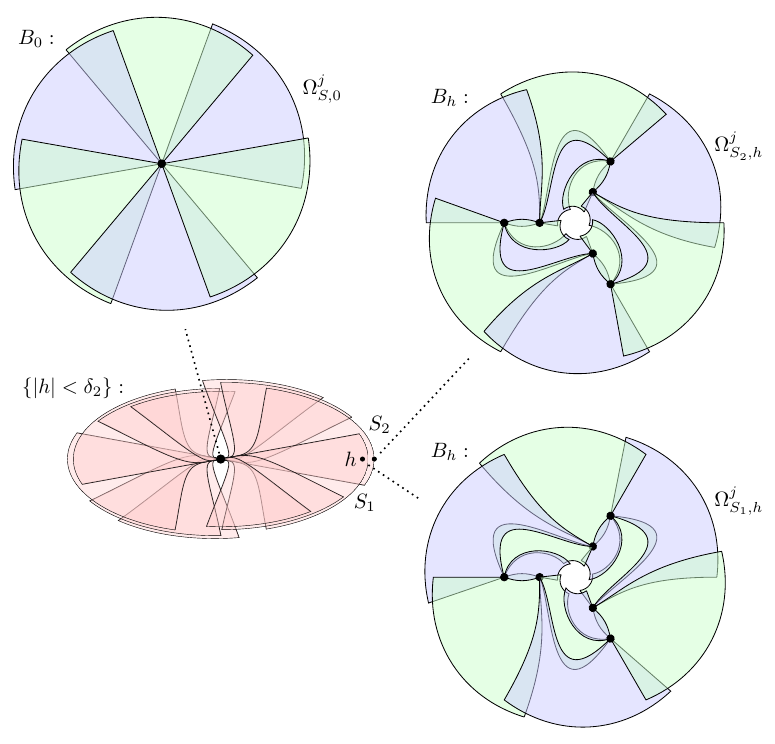}
	\caption{ Example of the domains of Theorem~\ref{thm:sectorial} in the case $p=3$, $k=1$, $s=0$,
		for a model ${\bf X}_{\rm mod}=i(u^3+h)(\xi_1\dd{\xi_1}-\xi_2\dd{\xi_2})$.
		In the center of the figure: a covering of a small disc $\{|h|<\delta_2\}$ by a collection of cuspidal sectors $S$ (in pink).
		For each sector $S$ and $h\in S$, the leaf $B_h=\{h={\rm const}\}\cap\{|\xi|<\delta_1\}$, which in the coordinate $\xi_1$ has the form of an annulus,
		is covered by $4kp$ Lavaurs domains $\Omega_{S,h}^j=\Omega_S^j\cap B_h$ attached to the fixed points $\Fix(\phi^{\circ p})\cap B_h$.
		When $h$ belongs to several sectors $S$ the associated coverings differ.
		The zero level set $B_0$ consists of two irreducible components, the figure shows the covering of only one of them.}
	\label{figure:sectors}
\end{figure}

\subsection{The moduli space for the analytic classification}

We express the modulus of the analytic classification as a countable collection of ``cocycles'' of transition maps on certain intersections of the covering.
Namely, to each cuspidal sector $S$ in the $h$-plane and its associated normalizing cochain $\{\Psi_{\Omega^j_{S}}\}_{j=1,\ldots 4kp}$ on the $4kp$ Lavaurs domains $\{\Omega_S^j\}$, we  associate  a set of \emph{transition maps} on the intersections of two subsequent outer/inner domains:
\[ \left\{\psi_S^{j,i}:=\Psi_{\Omega^j_{S}}\circ\Psi_{\Omega^i_{S}}^{\circ(-1)}\right\}_{i,j},\quad \text{defined on } \Omega^i_S\cap\Omega^j_S, \]
preserving ${\bf X}_{\rm mod}$ and possessing a $(\sigma,\Lambda)$-equivariance property.
The equivalence class of the set of transition maps
modulo conjugation by cochains of flow maps \eqref{eq:cochainofflowmaps}
\[\psi_S^{j,i}\simeq \exp(C_{\Omega^j_{S}}(h){\bf X}_{\rm mod})\circ\psi_S^{j,i}\circ\exp(-C_{\Omega^i_{S}}(h){\bf X}_{\rm mod}),\]
is then called a \emph{cocycle}.


\begin{thm}[\!\!\cite{stolo-klimes}, analytic classification]\label{thm:analytic}
	Let $\phi,\phi'=\Lambda\xi+\hot$ be two analytic $\sigma$-reversible germs of formal type \textnormal{(b)} or  \textnormal{(c)} of Theorem~\ref{thm:1} in the prenormal form \eqref{eq:prenormalphi}, both with the same model $\phi_{\rm mod}$ \eqref{eq:Xmodel}. The following are equivalent:
	\begin{enumerate}
		\item $(\phi,\sigma)$ and $(\phi',\sigma)$ are analytically conjugated by an element of $\Diff_{\id}(\cc^2,0)$.
		\item $(\phi,\sigma)$ and $(\phi',\sigma)$ are analytically conjugated by an element of
		\[\Diff_{\id}^{h}(\cc^2,0)=\{\psi\in \Diff_{\id}(\cc^2,0),\ h\circ\psi=h\}.\]
		\item For every cuspidal sector $S$ of Theorem~\ref{thm:sectorial} their associated cocycles $\{\psi_S^{j,i}\}$, $\{{\psi'}_S^{j,i}\}$ agree.
		\item For one cuspidal sector $S$ of Theorem~\ref{thm:sectorial} their associated cocycles $\{\psi_S^{j,i}\}$, $\{{\psi'}_S^{j,i}\}$ agree.
	\end{enumerate}
\end{thm}

In order to obtain the modulus of analytic equivalence with respect to conjugation by general transformations in $\Diff(\cc^2,0)$, one has to further consider the action
on the cocycles of the group
\[\cL Z^{\sigma,\Lambda}({\bf X}_{\rm mod})=\{\psi\in\Diff(\cc^2,0):\ \psi=\sigma\psi\circ\sigma=\Lambda^{-1}\psi\circ\Lambda,\ \psi^*{\bf X}_{\rm mod}={\bf X}_{\rm mod}\},\]
which by Theorem~\ref{thm:1} is a subgroup of the group $\{\xi\mapsto e^{\frac{\pi i r}{kp+2s}}\sigma^r\xi,\ r\in\zz_{2kp+4s}\}$.

\begin{remark}
	The restriction of $\phi$ to either irreducible component of the zero level set $\{h(\xi)=0\}$	 is a parabolic diffeomorphism of $(\cc,0)$ whose   Birkhoff--\'Ecalle--Voronin modulus (see Appendix \ref{ecalle-voronin}) agrees with the corresponding restriction of the classifying  cocycle $\{\psi_S^{j,i}\}$  for each of the cuspidal sectors $S$ in the $h$-plane. In particular, this implies that the moduli space is indeed infinite-dimensional. 
\end{remark}

Additionally, we show that, on each of the cuspidal sectors $S$ of Theorem~\ref{thm:sectorial} there exists  a canonical sectorial realization $\mu_S(h)$ of the formal invariant $\hat{\mu}(h)$ of Theorem~\ref{thm:1}, and therefore also a {\bf sectorial normal form} $(\phi_{{\rm nf},S},\sigma)$,
\[\phi_{{\rm nf},S}(\xi)=\Lambda\exp\big(\tfrac1p{\bf X}_{{\rm nf},S}\big)(\xi),\qquad {\bf X}_{{\rm nf},S}=\tfrac{c\,h^sP(u,h)}{1+c\,\mu_S(h)P(u,h)}\big(\xi_1\tdd{\xi_1}-\xi_2\tdd{\xi_2}\big).\]
Theorem~\ref{thm:sectorial} can then be reformulated with the sectorial normal form $\phi_{{\rm nf},S}$ in place of the model $\phi_{\rm mod}$.
Both the sectorial normal form $\phi_{{\rm nf},S}$ and the canonical normalizing cochains are asymptotic to their formal counterparts $\hat\phi_{\rm nf}$ and $\hat\Psi$ of Theorem~\ref{thm:1}, and can be interpreted as their ``sums'' of a sort. 

\subsection{Antiholomorphic parabolic reversible diffeomorphisms} 
\label{sec:1-anti}

Let $(\chi,\tau)$ be a pair of an \emph{antiholomorphic diffeomorphism} $\chi$, i.e. $\big(\xi\mapsto\ov{\chi(\xi)}\big)\in\Diff(\cc^2,0)$,
and of a \emph{holomorphic reflection} $\tau\in\Diff(\cc^2,0)$ such that
\begin{equation}\label{eq:reversible2}
\tau^{\circ 2}=\id,\qquad  \tau\circ\chi\circ\tau=\chi^{\circ(-1)}.
\end{equation}
Then $\rho=\tau\circ\chi$ is an \emph{antiholomorphic involution} reversing $\chi$,
\[\rho^{\circ 2}=\id,\qquad  \rho\circ\chi\circ\rho=\chi^{\circ(-1)},\]
and the problem of classification of pairs $(\chi,\tau)$ with respect to holomorphic conjugation is equivalent to that  of pairs $(\tau,\rho)$, or of \emph{Moser--Webster triple involutions}
\begin{equation}\label{eq:MWtriple}
(\tau_1,\tau_2,\rho)=(\tau,\tau\circ\chi^{\circ 2},\tau\circ\chi),
\end{equation}
where two holomorphic reflections $\tau_1,\tau_2$ are intertwined by a third antiholomorphic involution $\rho$:
\begin{equation}\label{eq:intertwining1}
\tau_1\circ\rho=\rho\circ\tau_2.
\end{equation}

We shall assume that the reversible holomorphic diffeomorphism $(\phi,\tau)=(\chi^{\circ 2},\tau)$ satisfies Assumptions~\ref{assumptions}:
$\chi^{\circ 2p}\in\Diff_{\id}(\cc^2,0)$ for some $p\geq 1$, $\tau$ is a holomorphic reflection, and they have a first integral $H=H\circ\chi^{\circ 2}=H\circ\tau$ of Morse type.
Then up to a linear change of coordinates, $(\chi,\tau)$ and $H$ take the form:
\begin{equation}\label{eq:taurho}
\chi(\xi)=\Lambda^{\frac12}\sigma\ov\xi+\hot(\ov\xi),\  \tau(\xi)=\sigma\xi+\hot(\xi), \ H(\xi)=\xi_1\xi_2+\hot(\xi),
\end{equation}
where
\[\sigma=\left(\begin{smallmatrix}0& 1\\[3pt] 1&0\end{smallmatrix}\right),\qquad
\Lambda^{\frac12}=\left(\begin{smallmatrix}\lambda^{\frac12}& 0\\[3pt] 0&\lambda^{-\frac12}\end{smallmatrix}\right)=\ov\Lambda^{-1}.\]

Since $\phi=\chi^{\circ 2}$ is a holomorphic diffeomorphism of $(\cc^2,0)$ reversed by both $\tau$ and $\rho$,
the classification of pairs $(\chi,\tau)$ is a priori a refinement of that of holomorphic pairs $(\phi,\tau)$ with an additional antiholomorphic symmetry.

\begin{thm}\label{thm:antiholomorphicclassification}
 Let two pairs $(\chi,\tau),\ (\chi',\tau')$ of the form \eqref{eq:reversible2} with  $(\chi^{\circ2},\tau)$,  $(\chi'^{\circ2},\tau')$  satisfying
Assumptions~\ref{assumptions}. 
\begin{enumerate}
	\item The two pairs are analytically (resp. formally) conjugated by a tangent-to-identity transformation if and only if the new two pairs $(\chi^{\circ 2},\tau)$, $({\chi'}^{\circ 2},\tau')$ are.
	\item  The two pairs are  analytically
(resp. formally) conjugated by a general transformation if and only if $(\chi^{\circ 2},\tau)$, $({\chi'}^{\circ 2},\tau')$ are analytically conjugated by a transformation with real linear part.
\end{enumerate}
\end{thm}

So by virtue of Theorem~\ref{thm:antiholomorphicclassification}, the formal normal form $(\hat\phi_{\rm nf},\hat\tau_{\rm nf})$ of Theorem~\ref{thm:1} for $(\chi^{\circ 2},\tau)$, provides in fact also a formal normal form $(\hat\chi_{\rm nf},\hat\tau_{\rm nf})$, namely
\begin{equation}\label{eq:chinormalform1}
\hat\chi_{\rm nf}(\xi)=\sigma\hat\rho_{\rm nf},\qquad
\hat\rho_{\rm nf}=\exp(-\tfrac{1}{2p}\hat{\bf X}_{\rm nf})(\Lambda^{-\frac12}\bar\xi), \qquad
\hat\tau_{\rm nf}(\xi)=\sigma\xi,
\end{equation}
where $\hat{\bf X}_{\rm nf}$ is as in Theorem~\ref{thm:1}\,(o)--(b), and satisfies $\hat{\bf X}_{\rm nf}=-\ov{(\Lambda^{\frac12})^*\hat{\bf X}_{\rm nf}}$. 

However it is sometimes more convenient to linearize the antiholomorphic involution $\rho=\tau\circ\chi$ instead of $\tau$,
which leads to a more symmetric formal normal form for the Moser--Webster triple \eqref{eq:MWtriple} (see \cite[sect. 1.3]{stolo-klimes}).

%

\section{Parabolic $n$-submanifolds 
}\label{gamma=.5}
\setcounter{thm}{0}\setcounter{equation}{0}
We now turn to a real analytic submanifold $M$ with the parabolic CR singularity. Assume further that its CR singularity has the maximum of dimension $n-1$ in $M$. Then in suitable holomorphic coordinates, the CR singular set is  given by
\begin{equation}\label{dim=n-1}
z_n=(z_1+\bar z_1)^2p(z_1,\bar z_1,x_2,\dots, x_{n-1}),\quad  y_\alpha=(z_1+\bar z_1)q(z_1,\bar z_1,x_2,\dots, x_{n-1}),
\end{equation} where $p(0)=1$ and $q_2,\dots q_{n-1}$ are real valued.
Such a real analytic submanifold occurs naturally, when $M$ has a non-degenerate CR singularity at the origin and $$\RE(dz_1\wedge dz_2\wedge\cdots\wedge dz_n)|_{M}=0.
$$ 
When $n=2$, such a real surface is called \emph{real Lagrangian} introduced by Webster~\ci{We92} and he showed that $M$ is always formally equivalent to $Q_{1/2}\colon z_2=(z_1+\bar z_1)^2$ under a formal holomorphic symplectic transformation of $\cc^2$ that preserves the holomorphic $2$-form $dz_1\wedge dz_2$.  It was observed by Moser that the linearized equation for this normal form problem has divergent solutions. Based on this observation, Gong~\ci{Go96} showed the existence    of real analytic Lagrangian surfaces with CR singularity that are not equivalent to $Q_{1/2}$ even under the larger group of biholomorphisms. It turns out the holomorphic classification of real analytic Lagrangian surfaces that are formally equivalent to $Q_{1/2}$ has an infinite dimensional moduli space. This classification was achieved by Ahern-Gong~\cite{AG09} and it relies on an important work of Voronin~\cite{V93}, which we now describe.

\subsection{A basic theorem of Voronin}
\begin{thm}[\!\!\cite{V93}]
Let $f$ be a local biholomorphism of $\cc^n$ fixing the origin such that the set of fixed points $f$ is a smooth complex hypersurface in $\cc^n$ and the linear part of $f$ is not diagonalizable. Then there exists a formal biholomorphic map $\Phi_0$ such that
$
\Phi_0^{-1}\sigma\Phi_0$ is the linear map $\hat f(x,y,z)=(x,y+4x,z)
$
where $(x,y,z)$ are the coordinates of $\cc\times\cc\times\cc^{n-2}$. 
\end{thm}

In fact, Voronin showed  a much stronger result holds, where the formal biholomorphic map $\Phi_0$ is replaced by a semi-formal map $\Phi$ of which the formal power series expansion is $\Phi_0$. Let us describe the semi-formal maps. 
Let $x,y\in\cc$ and $\zeta=(\zeta_2,\ldots,
\zeta_{n-1})\in\cc^{n-2}$.
A power series
 $h(x,y,\zeta)=\sum_{k\geq0}h_j(y,\zeta)x^j$ is called  {\it semi-formal\,} in $x$,
 if  all $h_j$ are
holomorphic in $(y,\zeta)$  on some fixed neighborhood $W$ of the
origin of $\cc^{n-1}$. A map $H=(H_1,\dots, H_n)$ is semi-formal if each $H_j$ is semi-formal.

We say that $S=V\times W$ is a sectorial domain, if $V$ is a
sector of the form $V_{\alpha,\beta,\epsilon}=\{x\in\cc\colon \arg
x\in(\alpha,\beta),0<|x|<\epsilon\}$ and $W$ is a neighborhood of
the origin in $\cc^{n-1}$, where $0<\beta-\alpha<2\pi$.
 A semi-formal power series
$G=\sum_{k=0}^\infty G_k(y,\zeta)x^k$ is called an {\it asymptotic
expansion} of a  holomorphic function $g$ on $V\times W$, denoted
by $g\sim G$ on $V\times W$,
 if there is a possibly smaller neighborhood $\widetilde W$ of $0\in\cc^{n-1}$ such that
 for each fixed  $N$
$$
\lim_{V\ni
x\to0}|x|^{-N}\bigl|g(x,y,\zeta)-\sum_{k=0}^NG_k(y,\zeta)x^k\bigr|=0$$
 holds  uniformly for $(y,\zeta)\in \widetilde W$.
 Analogously, we say
that a semi-formal map $\Phi$ is asymptotic to a holomorphic map
$H$ on $V\times W$, if each component of $\Phi$ is asymptotic to
the corresponding
 component of $H$.

The following is a fundamental result of Voronin. 
\begin{thm}[\!\!\cite{V93}]
\label{Vo}Let $(x,y,z)$ be the
coordinates of $\cc^n$. Let
 $f$ be a holomorphic map on $\cc^n$ of the form
\begin{equation}
\label{fxyz}
(x,y,z)\to (x+x^2p(x,y,z), y+x+xq(x,y,z),z+xs(x,y,z))
\end{equation}
where $q(0)=0=s(0)$.
Let $\alpha<\beta<\alpha+\pi$. There exist  $r$ depending  on
$\alpha,\beta$, and   a holomorphic map $B$ defined on
$\{x\colon\alpha<\arg x<\beta, |x|<r\}\times\Delta_r^{n-1}$ such that
 on $\{x\colon\alpha<\arg x<\beta, |x|<r\}\times\Delta_r^{n-1}$,
$B^{-1}fB=\hat f$ and $B$ admits the asymptotic expansion $\Phi$
which preserves $x=0$ and satisfies  $\Phi'(0)=\id$ and $\Phi|_{y=0}=\id$.
\end{thm}
Voronin gave a proof for $n=2$. The proof was adapted for case $n>2$ in the appendix of~\cite{AG09}.
 
We now describe applications of Voronin's   moduli space that
are relevant to  the above-mentioned two classification problems.

\subsection{A moduli space of parabolic submanifolds} 
We now study real analytic $n$-manifold $M$ in $\cc^n$ with Bishop invariant $\gamma=1/2$. We also assume that $M_{CRs}$ has maximum dimension $n-1$.
Put \begin{gather*}
\hat\tau_1\colon(x,y,\zeta)\to(-x,y+2x,\zeta),
  \quad\rho\colon(x,y,\zeta)\to(\ov x,-\ov y,\ov \zeta),
\\
\hat\tau_2=\rho\hat\tau_1\rho\colon(x,y,\zeta)\to(-x,y-2x,\zeta),\\
\quad\hat\sigma=\hat\tau_2\hat\tau_1\colon(x,y,\zeta)\to(x,y+4x,\zeta). 
\end{gather*}

Let $V_1=V_{\epsilon,\delta}=V_{-\epsilon,\frac{\pi}{2}+\epsilon,\delta}$.  Put
$V_j=\sqrt{-1}^{1-j}V_1$, $\Delta_\delta=\{t\in\cc\colon
|t|<\delta\}$. Assume that $0<\epsilon<\frac{\pi}{4}$. In particular, $S_{j\, j+1}=(V_{j}\cap
V_{j+1})\times \Delta_\delta^{n-1}$ are disjoint for $j=1,2,3,4$.
Let $\mathcal H$ be the set of $H=\{H_{1\, 2},
H_{2\, 3},H_{3\, 4}, H_{4\, 1}\}$ satisfying the following: $H_{j\,
j+1}$ is defined on $S_{j\, j+1}$ and 
\begin{gather} 
 H_{1\,2}^{-1}=\rho H_{1\,2}\rho,  
\ H_{4\,1}^{-1}=\rho H_{2\,3}\rho, \ H_{3\,4}=\hat\tau_j
H_{1\,2}\hat\tau_j,
\\ H_{j\,j+1}\sim \id, \quad \text{on $S_{j\, j+1}$}, 
\end{gather}
 in which
the positive numbers $\epsilon$ and $\delta$
 depend on $H$. Note that we define $H_{4\,5
}=H_{4\, 1}$.
By an abuse of notation,  we say that an identity holds on a sectorial
domain such as $V_{\alpha,\beta,\epsilon}\times W$
means that it holds on   $V_{\alpha+\delta,\beta-\delta,\epsilon'}
\times\Delta_{\epsilon'}^{n-1}$ for any $\delta>0$ and some $\epsilon'$ dependent of $\delta$.

We say that $H,\widetilde H$ are {\em equivalent} and write
$\widetilde H\sim H$, if there exist a semiformal map $\Psi$ and
biholomorphic maps $G_j=G_{j+4}$,
 defined on  $S_j'\equiv  \sqrt{-1}^{1-j}V_{\epsilon',\delta'}\times\Delta_{\delta'}^{n-1}
 $ (for some positive $\epsilon',\delta'$) or on $S_{j+2}'$
  and satisfying
\begin{gather}
\label{h12e-}
\begin{array}{l}
\widetilde H_{j\,j+1}=G_j^{-1}H_{j\,j+1}G_{j+1},\ j=1,\dots,4;
\quad \text{or}
\vspace{.75ex}\\
\widetilde H_{j+2\,j+3}=G_j^{-1}H_{j\,j+1}G_{j+1},\ j=1,\dots,4;
\end{array}
\\
G_2=\rho G_1\rho, \quad G_4=\rho G_3\rho, \quad
G_{j+2}=\hat\tau_kG_j\hat\tau_k; \label{g2rg-}
\\
G_j\sim\Psi, \quad \text{on $S_j'$ or on
$S_{j+2}'$};\label{h12e+-}
\\
\Psi\colon (x,y,\zeta)\to (a(x,\zeta)x,ya(x,\zeta)+b(x,\zeta),c(x,\zeta)),
\label{h12e++-} 
\end{gather}
where $a,b,c$ are semi-formal in $x$,
$a(0)\neq0, b(0)=0, c(0)=0$, and $\zeta\to c(0,\zeta)$ is biholomorphic. Such a mapping $\Phi$ is called \emph{semi-affine} 
by Voronin.
Note that $\Psi=\rho\Psi\rho=\hat\tau_j\Psi\hat\tau_j$. In particular,
$a(0)$ is real. The first case in \re{h12e-} is for $a(0)>0$ and its second case is for $a(0)<0$.

\subsection{Parabolic submanifolds  under volume-preserving maps}
 A formal map of $\cc^n$ that preserves  $\omega=dz_1\wedge \cdots \wedge dz_n$ is called unimodular.
Let ${\mathcal M}^{par}$ be the set of real analytic $n$-manifolds $M$ in
$\cc^n$,
 of which  complex tangents
form a germ of real analytic set of dimension $n-1$ at the origin,
while the origin is a parabolic complex tangent of $M$.  Denote by
${\mathcal M}^{par}_\omega$   the set of $M\in{\mathcal M}^{par}$ satisfying $\RE
\omega|_M=0$.   Let $\mathcal
H_{\hat\omega}$ be the set of $H\in\mathcal H$ satisfying the
additional condition
$$
H_{j\,j+1}^*{\hat\omega}={\hat\omega}, \quad
{\hat\omega}:=xdx\wedge dy\wedge d\zeta_{2}\wedge\cdots\wedge d\zeta_{n-1}.
$$
For $H,\widetilde H\in \mathcal H_{\hat\omega}$, we denote
$\widetilde H\sim H$, if there are $G_j,\Psi$ satisfying
\re{h12e-}-\re{h12e++-} and
$G_j^*{\hat\omega}={\hat\omega}.$
 Note that $\Psi^*{\hat\omega}={\hat\omega}$.
  Denote by $\mathcal H/{\, \sim}$ and $\mathcal H_{\hat\omega}/{\,
\sim}$ the corresponding sets of equivalence classes.

\

 Denote by ${\mathcal M}^{par}/{\,\sim}$ the set of
holomorphic equivalence classes in ${\mathcal M}^{par}$, and by $\mathcal
M_\omega/{\,\sim}$ the set of equivalence classes in $M_\omega$
under unimodular holomorphic maps.

The following theorem solves the two classification problems
mentioned early in this section.
\begin{thm}[\!\!\cite{AG09}] Each $M\in{\mathcal M}^{par}$
is formally biholomorphic  to
 $$
 Q\colon z_n=(z_1+z_1)^2, \quad\IM z_2
=\cdots=\IM z_{n-1}=0
$$
and each $M\in{\mathcal M}^{par}_\omega$ is equivalent to $Q$ under some
formal map preserving $\omega$. There are one-to-one correspondence between
${\mathcal M}^{par}/{\,\sim}$ and $\mathcal H/{\,\sim}$ and one-to-one
correspondence between ${\mathcal M}^{par}_\omega/{\,\sim}$ and $\mathcal
H_{\hat\omega}/{\,\sim}$; moreover, $\mathcal H/{\,\sim}$ and
$\mathcal H_{\hat\omega}/{\,\sim}$ are of infinite dimension. 
\end{thm}

\subsection{The original Voronin moduli space}
The Voronin's moduli space for the classification problem is more involved.    Suppose that $S_j=V_j\times\Delta_{\delta}^{n-1}$ are sectorial domains covering $\Delta_{\delta}\times\Delta_\delta^{n-1}$. Suppose that there are   normalized sectorial normalizing maps $H_j\sim\Phi$ on $S_j$ such that $H_j^{-1}f H_j=\hat f$ where $f$ has the form \re{fxyz} and $\hat f$ is the linear part of $f$.  Take an affine semi-formal biholomorphism $G(x,y,\zeta)=(xa(x,\zeta),ya(x,\zeta)+b(x,\zeta),c(x,\zeta))$ of the form \re{h12e++-}. Note that $G$ commutes with the linear map $\hat\sigma$. Thus $\tilde H_j= H_j G$ are asymptotic to the same map $\Phi G$, while $\tilde H_j^{-1}f\tilde H_j=\hat f$ holds on the new sectorial domain $\tilde S_j=G'(0)^{-1}S_j$. In such a way, one gets a large family  sectorial domains on which $f$ are normalized. The projection of $\tilde S_j$ onto the $x$-plane is a rotation $a(0)V_j$ of $V_j$.  To deal with this kind of rotation, Voronin introduced the \emph{narrowing} (refinement) of any finite covering by sectorial domains $\{S_j\}$ and refinement of transitions $\{H_{jk}\}$ and their equivalent relations ( p. 199) as follows. 

  Let $S_1,\dots, S_N$ be sectorial domains covering $\Del_\del^*\times\Delta_\delta^{n-1}$ and let $H_{j,j+1}\sim I$ be semi-formal transformation on $S_{j,j+1}$. Define $\{S_j, H_{j,j+1}\}_{j=-\infty}^{\infty}$ such that it has period $N$, i.e.   $V_{j+N}=V_j$ and $H_{j+N,j+N+1}=H_{j,j+1}$ for all $j\in\zz$. Let $\mathcal V$ be the set of all the such pairs $\{S_j,H_{j,j+1}\}$ that has a finite period. One says that
$\{\hat S_\alpha,\hat H_{\alpha,\alpha+1}\}_{\all=-\infty}^{\infty}$ that has period $kN$ is a narrowing of $\{V_j,H_{j,j+1}\}$, if  $\hat S_{\all}\subset V_{j}$ and $\hat H_{\all,\all+1}=H_{j,j+1}$ for some $j$. Write $\{S_j, H_{j,j+1}\}\sim \{S'_j, H'_{j,j+1}\}$ if they are in $\mathcal V$ and have narrowing  $\{\hat S_\all, \hat H_{\all,\all+1}\}$ and $\{\hat S'_\all,\hat H'_{\all,\all+1}\}$ satisfying
$\hat H'_{\all',\all'+1}=G_{\all}^{-1}\hat H_{\all,\all+1}G_{\all+1}$ for all $\alpha$, where all $G_\all$ are asymptotic to the same affine semi-formal map $\Psi$ of the form \re{h12e++-}.

\begin{thm}[\!\!\cite{V93}]
 The  holomorphic equivalent classification of holomorphic mappings \rea{fxyz} is given by $\mathcal V/{\,\sim}$.
\end{thm}

\begin{rem}In~\cite[p.~9]{AG09}, it was stated {\em mistakenly} that when $M\subset\cc^n$ has a non-degenerate complex tangent at the origin, there is a quadratic change of coordinates that transforms $M$ into 
$$
z_n=az_1\bar z_1+bz_1^2+c\bar z_1^2+O(2), \quad y_\alpha=O(2),\quad 1<\all<n.
$$
This is correct for $n=2$. When $n>2$, the correct assertion is the following: $M$ can be transformed into
\begin{align*}
z_n&=z_1\bar z_1+\gamma(z_1^2+\bar z_1^2)+O(3),\  y_\alpha=O(2),\quad 1<\all<n, \ \gamma\in[0,\infty)\setminus\{1/2\};\quad \text{or}
\\
z_n&=z_1^2+\bar z_1^2+O(3),  y_\alpha=O(2),\ 1<\all<n, \ \gamma=\infty;\quad
\text{or}\\
z_n &=z_1\bar z_1+\frac{1}{2}(z_1^2+\bar z_1^2)+i\delta_n x_2(z_1 + \overline z_1)+O(3),  y_\alpha=O(2),\  1<\all<n, \ \gamma=1/2.
\end{align*}
Here $\delta_n=0,1$. When $\delta_n=1$ occurs,  the set of complex tangent points of $M$ has codimension $2$. Therefore, this mistake does not affect other results in the paper since the authors assumed that the set of complex tangent points of $M$ has codimension one. Finally, we should mention that the Moser-Webster involutions $\{\tau_1,\tau_2,\rho\}$ still apply for the case $\delta_n=1$.
\end{rem}

\section{Codimension-2 real submanifolds with CR singularity}\label{sect:highcod}
\setcounter{thm}{0}\setcounter{equation}{0}

In the previous sections, we discuss CR singular $(n+1)$-submanifolds in $\mathbb C^{n+1}$ with $n\geq1$.  
Recall the CR singularity cannot occur for codimension-1 submanifolds. This leads naturally to the case of codimension-2 CR singular submanifolds in $\cc^{n+1}$ for $n>1$, which was initially studied by Dolbeault-Tomassini-Zaitsev~\cite{DTZ05, DTZ10, DTZ11}.
The formal normal forms for codimension-two real submanifolds  in $\cc^{n+1}$ have been studied
by Coffman~\ci{Co09}, Huang-Yin~\cite{HY09b, HY16, HY17}, Burcea~\cite{Bu13}, Lebl-Noell-Ravisankar~\cite{LNR17} and Fang-Huang~\ci{huang-fang-gafa}. For exposition, we follow the latter.


To be more detailed, we assume that $p\in M$ is a CR singular point
of 
a smooth real submanifold $M$ in $\mathbb C^{n+1}$. We write $(z_1,\cdots,z_n,w)$ for the coordinates of
${\mathbb C}^{n+1}$. After a holomorphic change of coordinates, we
assume that $p=0,\ T^{(1,0)}_pM=\{w=0\}$. Then $M$ near $p=0$ is the
graph of a function of the form:
\begin{equation}\label{001}
	w=F(z,\-{z})=q^{(2)}(z,\ov{z})+o(|z|^2),
\end{equation}
where $q^{(2)}(z,\ov{z})$ is a polynomial of degree two in
$(z,\ov{z})$. In the classical  case, namely,  $n+1=2$, after a
holomorphic change of variables, we can always make
$q^{(2)}(z,\ov{z})$-real-valued. However, this is no longer the case
for $n+1\ge 3$. Indeed, after a simple holomorphic change of
coordinates, if needed, we can  write
\begin{equation} \label {0011}
	q^{(2)}(z,\ov{z})=2\RE{(z\cdot \mathcal A\cdot z^t)}+z\cdot \mathcal B\cdot \ov{z}^t
\end{equation}
with $\mathcal A$ and $ \mathcal B$ being two $(n\times
n)$-matrices. Suppose that $z=\tilde{z}\cdot
P+\vec{a}\tilde{w}+O(|(z,w)|^2);\ w=\mu \tilde{w}+z\cdot b^t+
O(|(z,w)|^2)$  is a holomorphic transformation preserving the form
as in (\ref{001}) and (\ref{0011}). Then $b=0,\mu\not =0$, and $P$
is an $(n\times n)$-invertible matrix. Moreover, if $(M,0)$ is
defined in the new coordinates by
\begin{equation}\label{002}
	\tilde{w}=\tilde{q}^{(2)}(\tilde{z},\ov{\tilde{z}})+o(|\tilde{z}|^2),
\end{equation}
with $\tilde{q}^{(2)}(\tilde{z},\ov{\tilde{z}})=2\RE{\tilde{z}\cdot \widetilde{\mathcal A}\cdot
	\tilde{z}^t}+\tilde{z}\cdot \widetilde{\mathcal B}\cdot \ov{\tilde{z}}^t.$ Then

\begin{equation}\label{00111}
	\widetilde{\mathcal B}=\frac{1}{\mu} P\cdot \mathcal B\cdot \ov{P}^t,\ \
	\widetilde{\mathcal A}=\frac{1}{\ov\mu} P\cdot\mathcal A\cdot {P}^t.
\end{equation}

When there do not exist
a $\mu\not =0$ and an invertible $P$ such that $\frac{1}{\mu} P\cdot
\mathcal B\cdot \ov{P}^t$ is Hermitian, one can never make
$\tilde{q}^{(2)}(z,\ov{z})$ real-valued. Also notice that the
non-degeneracy of the matrix ${\mathcal B}$ is a holomorphic
invariant property.
More general, we make   the following definition: 
\begin{defn}
	\begin{enumerate} 
\item A smooth real-codimension two submanifold $M\subset\mathbb C^{n+1}$ is {\it non-minimal} at CR points in the sense of Tumanov~\cite{Tum88}, if each CR point of $M$ is contained in a CR smooth submanifold of $\mathbb C^n$ that is contained in $M$ and has CR dimension $n-1$.
	\item  $M$ is said to have a {\it non-degenerate CR singularity} at $p$ if  there is a
	holomorphic change of variables such that in the new coordinates,
	$p=0$, $M$ is defined by an equation of the form as in (\ref{001})
	and (\ref{0011}) with $\det{\mathcal B}\not =0$. If there is a
	holomorphic change of variables such that ${\mathcal B}$ is a
	definite Hermitian matrix, we call $p$ a {\it definite  CR singular point}
	of $M$.
	\item Let $M$ be a real-codimension two real submanifold with
	$p\in M$ a CR singular point. We say that $M$ is {\it quadratically
	flattenable} if there is a change of coordinates such that in the new
	coordinates, $p=0$, $M$ near $p=0$ is defined by an equation of the
	form as in (\ref{001}) with $q^{(2)}(z,\ov{z})$ real-valued. One
	says that $M$ can be {\it holomorphically flattened} at $p$ if there is a
	holomorphic change of variables such that in the new coordinates,
	$p=0$, $M$ is defined by an equation of the form as in (\ref{001})
	with $\IM\left(F(z,\-{z})\right)\equiv 0.$
\end{enumerate}
\end{defn}
In  Dolbeault-Tomassini-Zaitsev~\cite{DTZ10, DTZ11} and Huang-Yin~\cite{HY16, HY17} the starting point is to define a generalized
notion of the  Bishop non-degeneracy and generalized Bishop
invariants at a CR singular point $p$. For that purpose, one needs
to assume that $M$ near $p$ is quadratically flattenable. However,
in the setting considered in Dolbeault-Tomassini-Zaitsev~\cite{DTZ10, DTZ11} and Huang-Yin~\cite{HY16, HY17}, $M$ is always CR
non-minimal at its CR points. This raises a natural question by Zaitsev (see~\cite{huang-fang-gafa})
to understand the implication of CR non-minimality to the quadratic
flattenability of $M$  near a CR singular point.

\begin{thm}[\!\!\cite{huang-fang-gafa}]
\label{003} Let $M$ be a real codimension-two smooth submanifold in
	${\mathbb C}^{n+1}$ with $p\in M$ a  non-degenerate  CR singular
	point.
	Assume that $M$ is CR non-minimal at its CR points near $p$. Then
	$M$ is quadratically flattenable.
\end{thm}


\medskip

Let $M\subset {\mathbb C}^{n+1}$ be a codimension-two real
submanifold with $0\in M$ a non-degenerate CR singular point, defined by
 (\ref{001}) (\ref{0011}) with ${\mathcal B}$ a
Hermitian matrix. When ${\mathcal B}$ is  definite, then by the
classical Takagi theorem, we can further make
$q^{(2)}(z,\-{z})=\sum_{j=1}^{n}\left(|z_j|^2+\lambda_j(z_j^2+\-{z_j^2})\right),$
where $0\leq \lambda_1\leq\cdots, \lambda_n<\infty.$ The set
$\{\lambda_1,\cdots,\lambda_n\}$ is called the set of generalized
Bishop invariants of $M$ at the CR singular point. $\lambda_j$ is
called an elliptic, parabolic or hyperbolic Bishop invariant, if
$0\leq \lambda_j<\frac{1}{2}$,  $ \lambda_j=\frac{1}{2}$ or
$\lambda_j>\frac{1}{2}$. This terminology coincides with the
classical definition of Bishop~\cite{Bi65} when $n+1=2$.
However, when ${\mathcal B}$ is not a definite matrix, we cannot,
in general, simultaneously diagonalize ${\mathcal A}$ and ${\mathcal
	B}$. In the case of $n+1=3$, Coffman~\cite{Co09} gave a list of the forms that
the pair $\{{\mathcal A},{\mathcal B}\}$ can be transformed to.  Two cases in his list
are geometrically quite special, in which the corresponding
quadratic term takes one of the following forms after a holomorphic
change of coordinates:
\begin{equation}
	\label{P1}
	q^{(2)}=|z_1|^2+|z_2|^2+\frac{1}{2}(z_1^2+\-{z_1^2})+\frac{1}{2}(z_2^2+\-{z_2^2});\
	\hbox{or}
\end{equation}
\begin{equation} \label{M1}
	q^{(2)}=|z_1|^2-|z_2|^2+\lambda(z_1^2+\-{z_1^2})+\lambda(z_2^2+\-{z_2^2}),\
	\ \lambda\ge \frac{1}{2}.
\end{equation}
In the case of (\ref{P1}), the two generalized Bishop invariants of
the CR singular point at the origin are both  parabolic.  Consequently, the set of CR singular points may have
real dimension $n=2$, which does create a lot of problems for the
geometric studies of $M$ near $0$.

To explain the speciality of (\ref{M1}), we recall a definition from~\cite{HY16}. Let $(M,p)$ be a codimension two real submanifold in
${\mathbb C}^{n+1}$ with $p\in M$ a CR singular point. We say
$(M,p)$ possesses an elliptic complex tangent direction if there is
an affine complex plane ${\mathcal H}$ that passes through $p$ and
is transversal to the complex tangent space of $M$ at $p$ such that
$M\cap {\mathcal H}$ is an elliptic Bishop surface inside ${\mathcal
	H}$ in the classical sense (see Section 2). Now,
a simple  algebraic computation  shows that a codimension two real
submanifold $M\subset {\mathbb C}^{n+1}$ with a non-degenerate
quadratically flattenable CR singular point at $p$ has no elliptic
directions at $p$ if and only if $n+1=3$ and after a holomorphic
change of variables sending $p$ to $0$,  $M$ near $p=0$ is defined
by an equation of the form as in (\ref{001}) with $q^{(2)}$ being
given by (\ref{M1}). 


The papers~\cite{HY16, HY17, huang-fang-gafa}
study the holomorphically
flattening problem near a CR singular point when  $M$ is real
analytic. 

Note that  a real surface in $\mathbb C^2$ at an elliptic or non-exceptional hyperbolic complex tangent can always be formally holomorphically flattened, as shown by Moser and Webster~\cite{MW83} and Moser~\cite{moser-zero}. Recall that an elliptic real analytic n-submanifold in $\mathbb C^n$ with non vanishing can be holomorphically flattened as a consequence of Moser-Webster normal form  \rt{NFofM}. An elliptic real analytic n-submanifold in $\mathbb C^n$ with a vanishing Bishop invariant in $\mathbb C^n$ holomorphically flattened by Huang and Krantz \cite{HK95} for $n=2$, and by Huang \cite{H98} for $n\geq 2$. There are real analytic hyperbolic surfaces in $\mathbb C^2$ that cannot be holomorphically flattened by \cite{MW83}; see \rp{non-flat-hyper}.

 Noticeably, when $n+1=3$, Huang and Yin~\cite{HY09b} showed that a real analytic submanifold defined by $w=|z_1|^2+|z_2|^2+\RE\{\sum a_{ij}z_1^iz_2^j\}+\sum_{i\geq 2,j\geq 2}b_{i\bar j}z_1^i\bar z_2^j$ can be formally holomorphically flattened if and only if $b_{i\bar j}=\overline{ b_{j\bar i}}$ for all $i,j$.
With non-minimal CR singularity, we have the following.
\begin{thm}[\!\!\cite{huang-fang-gafa}]
\label{005} Let $M$ be a real analytic real-codimension two submanifold in
	${\mathbb C}^{n+1}$ with $n\ge 2$ and with  $p\in M$ a
	non-degenerate CR singular point.
	Assume that $M$ is CR non-minimal at its CR points near $p$. Then
	$(M,p)$ can be holomorphically  flattened if $M$ has an elliptic
	direction at $p$. More precisely, $(M,p)$ can be holomorphically
	flattened if either $n+1\ge 4$ or  $n+1=3$ but $(M,p)$ is not
	holomorphically equivalent to a submanifold $(M',0)$ whose quadratic
	term takes the form  in (\ref{M1}).
\end{thm}
The result was proved in~\cite{HY16} when the $\cL B$ in \re{0011} is a positive (or negative)  definite hermitian matrix. 
The above theorem  has an
immediate application to the  study of the precise description of
the local hull of holomorphy of $M$.

\begin{cor}[\!\!\cite{huang-fang-gafa}] \label{44.44}
	Let $M$ be a real analytic real-codimension two submanifold in
	${\mathbb C}^{n+1}$ with $n\ge 2$ and with  $p\in M$ a
	non-degenerate CR singular point.
	Assume that $M$ is CR non-minimal at its CR points near $p$. Assume
	that either $n+1\ge 4$ or  $n+1=3$ but $(M,p)$ is not
	holomorphically equivalent to a submanifold $(M',0)$ whose quadratic
	term takes the form  in (\ref{M1}). Then there is  a real analytic
	Levi-flat hypersurface $\widehat{M}$, which has $M$ near $p$ as part of
	its real analytic boundary and is foliated by complex hypersurfaces
	shrinking down to $p$ along the normal direction of $M$ in $\widehat{M}$
	at $p$. Moreover, when $p$ is a definite CR singular point, then
	there is a small $\epsilon_0>0$ such that for any
	$0<\epsilon<<\epsilon_0$, $\widehat{M}\cap B_p(\epsilon)$ is a connected
	open
	piece containing the origin  of  the hull of holomorphy of $M\cap B_p(\epsilon_0)$. Here
	$B_p(\epsilon)$ denotes the ball  centered at $p$ of radius
	$\epsilon$.
	
\end{cor}

Theorem \ref{005} is contained in Huang-Yin [HY3] when $p$ is a
definite CR singular point with one of the generalized Bishop
invariants elliptic.

\section{Low dimension real submanifolds with CR singularity}\label{sect:lowdim}

\setcounter{thm}{0}\setcounter{equation}{0}

In this section, we discuss $m$-submanifolds in $\cc^n$ with CR singularity and $m<n$.
Coffman~\cite{Co06} showed that  any
$m$ dimensional real analytic submanifold in $\cc^n$ of  one-dimensional complex tangent
space at  a  CR singularity satisfying certain non-degeneracy conditions
is locally holomorphically equivalent to a unique algebraic submanifold,
provided $2(n+1)/3\leq m<n$.
Specifically, he proved the following.
\begin{thm}[\!\!\cite{Co06}]\label{thm-co-06}
Given $2(n+1)/3\leq m<n$, let $M^m$ be a real analytic $m$-submanifold in $\cc^n$ with a CR singularity at $0$ with $\dim T_0^{(1,0)}M=1$. Suppose that $M$ is a third order perturbation of $Q\subset\cc^n$ defined by
\begin{gather}
y_s=0,  \quad 2\leq s\leq m-1,\\
z_t=\bar z_1(x_{2(t-m+2)}+ix_{2(t-m+2)+1}), \quad t=m,\dots, n-2,\\
z_{n-1}=\bar z_1^2, \quad  z_n=\bar z_1(z_1+x_2+ix_3).
\end{gather}
Then $M$ is locally biholomorphic to $Q$.
\end{thm}
The model $Q$ arrives from two non-degeneracy conditions (3) and (4) in ~\cite{Co06}. A special case for $m=4$ and $n=5$, the formal normalization was achieved early by  Beloshapka~\cite{Be97} and Coffman~\cite{Co97}. See also  Harris~\cite{Ha81} and Coffman~\cite{Co04} for $m=2$ and $n\geq 3$ for formal normal forms. 
Coffman~\ci{Co10} further developed \rt{thm-co-06} into unfolding CR singularity by studying parameterized families of CR singular real analytic submanifolds. The reader is referred to~\ci{Co10} for the results and references therein.

\section{Real manifolds with maximum CR singularity}\label{sect:maxCR}
\setcounter{thm}{0}\setcounter{equation}{0}

Consider an  $n$-dimensional real analytic  submanifold $M$ in $\cc^n$ that has a CR singularity. A CR singularity is called {\it maximum} in~\ci{GS16}, if the real tangent space $T_0M$ is actually a complex subspace of $T_0\cc^n$. Thus  $p=n/2$ must be an integer.  
   We are interested in
the normal form problem, the rigidity property,  and  the local analytic geometry  of such real analytic manifolds.

In  suitable holomorphic coordinates, a $2p$-dimensional real analytic submanifold $M$ in $\cc^{2p}$
 that has a complex tangent space of maximum dimension at the origin  is given by
\eq{mzpjintr}
M\colon z_{p+j}=E_j(z',\ov z'),
\quad 1\leq j\leq p,
\eeq
where $z'=(z_1,\ldots, z_p)$, $z=(z',z'')$,  and
$$
E_j(z',\ov z')=h_j(z',\ov z')+q_j(\ov z') + O(|(z',\ov z')|^3). 
$$
Moreover,    $h_j(z',\ov z')$ are homogeneous quadratic polynomials in $z',\ov z'$ without holomorphic or anti-holomorphic terms,
  and  $q_j(\ov z')$ are homogeneous quadratic polynomials in $\ov z'$.

To study $M$, we consider its complexification in $\cc^{2p}\times\cc^{2p}$ defined by
\begin{equation}
{\mathcal M}\colon
\begin{cases}
z_{p+i} = E_{i}(z',w'), & i=1,\ldots, p,
\\
w_{p+i} = \ov{ E_i}(w',z'),& i=1,\ldots, p.\\
\end{cases}
\nonumber
\end{equation}
 It is a complex submanifold of complex
dimension $2p$ with coordinates $(z',w')\in\cc^{2p}$.  Let $\pi_1,\pi_2$ be the restrictions of the projections $(z,w)\to z$
and   $(z,w)\to w$ to $\cL M$, respectively. 

  Our basic assumption is the following condition.

\medskip
\noindent
{\bf Condition B.}  $q(z')=(q_1(z'),\ldots, q_p(z'))$ satisfies $q^{-1}(0)=\{0\}.$
\medskip

When $p=1$, condition B corresponds to 
 the non-vanishing of  the Bishop invariant $\gaa$.
When $M$ is a {\it quadric}, i.e. all $E_j$ in \re{mzpjintr} are   quadratic  polynomials, our basic condition~B
is equivalent to $\pi_1$ being a $2^p$-to-1 branched covering, in which case $\pi_2$ is also a $2^p$-to-$1$ branched covering. We first recall the following.

\begin{exmp}Let $p>1$. Let $N_{\gaa,\e}$ be a perturbation
 of $Q_\gaa$ defined by
$$
z_{p+j}=z_j\ov z_j+\gaa_j\ov z_j^2+\e_{j-1} \ov z_{j-1}^3,\quad \epsilon_j\neq0,
\quad 1\leq j\leq p.
$$  
Then the identity map of $\cL M$ is the only deck transformation of  branched coverings $\pi_1,\pi_2$.
\end{exmp}

This example shows an essential difference between real analytic submanifolds  with maximum CR singularity and the ones  with minimum CR singularity.
Thus we introduce the following

{\bf Condition D.}  {\it $M$ satisfies condition $B$ and the branched   covering   $\pi_1$ 
of $\cL M$
admits the maximum $2^p$ deck transformations.}
\medskip

It was proved in that Condition D gives rise to two families of commuting involutions $\{\tau_{i1},\ldots, \tau_{i2^p}\}$
intertwined by the anti-holomorphic
involution $\rho_0\colon(z',w')\to(\ov w',\ov z')$ such that $\tau_{2j}=\rho_0\tau_{1j}\rho_0$ $(1\leq j\leq 2^p)$ are deck
transformations of $\pi_2$.   Further, 
  there is a unique set of $p$ generators for the deck transformations  of $\pi_1$, denoted by $\tau_{11},\ldots, \tau_{1p}$, such  that each $\tau_{1j}$   fixes a hypersurface in $\cL M$ pointwise.  Then
$\tau_1=\tau_{11}\circ\cdots\circ\tau_{1p}$
 is the unique deck transformation of which the fixed-point set   has the smallest dimension $p$.
 Let  $\tau_2=\rho_0\tau_1\rho_0$ and
  $
 \sigma=\tau_1\tau_2.
 $
 Then $\sigma$ is {\it reversible} by $\tau_j$ and $\rho_0$, i.e.
 $\sigma^{-1}=\tau_j\sigma\tau_j^{-1}$ and $\sigma^{-1}=\rho_0\sigma\rho_0$.

 As in the Moser-Webster theory,   the existence of such $2^p$ deck transformations    transfers the normal form problem for the real submanifolds into the normal form problem for the sets
  of involutions $\{\tau_{11},\ldots,\tau_{1p},
  \rho_0\}$.

Next, we make the following assumption.

 \medskip
 \noindent
{\bf Condition J.} {\it   $M$ satisfies condition D and  
   $\sigma'(0)$ is diagonalizable.}


\medskip


Note that the condition excludes the higher dimensional analogous complex tangency of {\it parabolic} type, i.e. of $\gaa=1/2$.

\subsection{Product quadrics}  The basic model for quadric manifolds   satisfying condition J is a product of  3 types of quadrics defined by
\begin{gather}\label{Qgs2}
 Q_{\gamma_e}\subset\cc^2 \colon z_{2}= (z_1+2\gaa_e\ov z_1)^2;\\
 Q_{\gamma_h}\subset\cc^2\colon z_{2}= (z_{1}+2\gaa_{h}\ov z_1)^2, \ 1/2<\gaa_h<\infty;\quad Q_\infty\colon z_2=z_1^2+\ov z_1^2;\\
 Q_{\gamma_s}\subset\cc^4\colon z_{3}= (z_1+2\gamma_s\ov z_{2})^2,
\quad z_{4}=( z_{2}+2
(1-\ov\gamma_{s})  \ov z_{1})^{2}.
\label{Qgs2+}
\end{gather}
Here $\gaa_s\in\cc$ and
\eq{0ge1}
0<\gamma_e<1/2, \quad
1/2<\gamma_h\leq\infty, \quad \RE
\gaa_s\leq1/2, \quad  \IM\gaa_s\geq0, \quad\gaa_s\neq0,1/2.
\eeq
Note that $Q_{\gaa_e}, Q_{\gaa_h}$ are elliptic and hyperbolic Bishop quadrics, respectively. However,  $Q_{\gaa_s}$ 
 is not holomorphically equivalent to a product of two Bishop surfaces and such a complex tangent is called \emph{complex}. 
  A product of the above quadrics
 will be called a {\em product of quadrics}, or a {\it product quadric}.
 We denote by $e_*,h_*, 2s_*$
  the number of elliptic, hyperbolic and complex type, respectively.   
 When $\sigma$ has $2p$ distinct eigenvalues, \cite[Thm.~1.1]{part2} gives  a complete classification of quadratic submanifolds of maximum deck transformations, showing that there are  quadratic manifolds which are not holomorphically equivalent to a product quadric.  

 We now describe   main geometrical and dynamical results for analytic higher order perturbations of product quadric.
We start with  a holomorphic normalization of a real analytic submanifold $M$ with the so-called abelian CR singularity.  This will be achieved by
studying an integrability problem on a general family of commuting biholomorphisms described below. 
 

\subsection{Normal form of commuting biholomorphisms}
\begin{defn}
Let $\cL F=\{F_1,\ldots, F_\ell\}$ be a finite family of germs of biholomorphisms of $\cc^n$
fixing the origin.
 Let $D_m$ be the linear part of $F_m$ at the origin.  We say that the family $\cL F$ is (resp. formally)  {\it completely integrable}, if there is a (resp. formal) biholomorphic mapping $\Phi$ such that
$\{\Phi^{-1}F_m\Phi\colon 1\leq m\leq \ell\}=\{\hat F_m\colon 1\leq m\leq \ell\}$ satisfies
\bppp
\item $\hat F_m(z)=(\mu_{m 1}(z)z_1,\ldots, \mu_{m n}(z)z_n)$ where $\mu_{mj}$ are germs of holomorphic (resp. formal) functions such that $\mu_{m j}\circ D_{m'}=\mu_{m j}$ for
 $1\leq m,m'\leq \ell$  and $ 1\leq j\leq n$.   In particular,  $\hat F_m$ commutes with $D_{m'}$ for all $1\leq m,m'\leq\ell$.
\item  For each $j$ and each $Q\in\nn^n$ with $|Q|>1$,  $\mu_m^Q(0)=\mu_{m j}(0)$ hold for all $m$ if and only if
 $\mu_{m}^Q(z)=\mu_{mj}(z)$ hold  for all $m$.
\eppp
\end{defn}
\begin{defn}\label{small divisors}
	We  say that 
	{\it the family $D$ is of Poincar\'e type}
	if there exist  constants $d>1$ and $c>0$ such that, for each $(j,Q)\in\{1,\ldots, n\} \times\nn^n$
	that satisfies
	$\mu_{m}^{Q}-\mu_{m j}\neq 0$ for some $m$,
	there exists $(i,Q')\in \{1,\ldots, \ell\}\times \nn^n$ such that
	$\mu_{k}^{Q'}=\mu_k^Q$ for all $1\leq k\leq \ell$,
	$\mu_{i}^{Q'}-\mu_{ij}\neq0$,
	and
	\gan
	\max\left(|\mu_{i}^{Q'}|,|\mu_{i}^{-Q'}|\right)>c^{-1}d^{|Q'|}, \quad
	\text{ $Q'-Q\in\nn^n\cup(-\nn^n$)}.
\end{gather*}
\end{defn}
A necessary condition for $\cL F$ to be formally completely integrable is that $F_1,\ldots, F_\ell$ commute pairwise.
We had the following main result.
\begin{thm}[\!\!\cite{GS16}] \label{lFba}
Let $\cL F$ be a family of finitely many
germs of biholomorphisms  at the origin.
If $\cL F$ is formally completely integrable and its linear part $\cL D$  has the Poincar\'e type, then it is holomorphically  completely integrable.
\end{thm}
Such a formal integrability condition can hold under some geometrical properties.
 For instance, for a single germ of real analytic hyperbolic area-preserving mapping, the result  was due to
Moser~\cite{moser-hyperbolic},
and for a single germ of reversible hyperbolic  holomorphic mapping $\sigma=\tau_1\tau_2$
of which $\tau_1$ fixes a hypersurface,
 this result was due to Moser-Webster~\cite{MW83}.
Such  results for commuting germs of vector fields were obtained in \cite{St00,  stolo-annals} under a   collective small divisors Brjuno-type condition.  

\subsection{Holomorphic normalization for the abelian CR singularity} We first recall the  {\it abelian} CR singularity introduced in~\ci{GS16}.
For a  product quadric $Q$ which    satisfies condition~J. 
   the deck transformations of
 $\pi_1$   are generated by $p$ involutions of which each fixes a hypersurface pointwise. We denote them by $T_{11}, \ldots, T_{1p}$.
 Let $T_{2j}=\rho T_{1j}\rho$.
It turns out that
  each $T_{1j}$ commutes with all $T_{ik}$
   except one,  $T_{2k_j}$ for some $1\leq k_j\leq p$. When we formulate $S_j=T_{1j}T_{2k_j}$ for $1\leq j\leq p$,
  the  $S_1, \ldots, S_p$ commute pairwise. 
 Furthermore, one can find 
 linear coordinates for the product quadrics such that the normal forms of
$S$, $T_{ij}$, $\rho$ are given by
\begin{align}
&S\colon\xi_j'=\mu_j\xi_j,\quad \eta_j'=\mu_j^{-1}\eta_j,\quad 1\leq j\leq p;\label{sxij-mv}\\
\label{rSxi-mv}
&S_j\colon\xi_j'=\mu_j\xi_j,\quad \eta_j'=\mu_j^{-1}\eta_j,\quad \xi_k'=\xi_k,\quad \eta_k'=\eta_k,
\quad k\neq j;\\
\label{rTij-mv}
&T_{ij}\colon\xi_j'=\la_{i j}\eta_j,\quad\eta_j'=\la_{i j}^{-1}\xi_j,\quad\xi_k'=\xi_k,\quad
\eta_k'=\eta_k, \quad k\neq j; 
\\ 
\label{rRho-mv}
& \rho\colon \left\{\begin{array}{ll}
(\xi_e',\eta_e',\xi_h',\eta_h')=
(\ov\eta_e,\ov\xi_e,\ov\xi_h,\ov\eta_h),\vspace{.75ex}
\\
(\xi_{s}', \xi_{s+s_*}',\eta_{s}',\eta_{s+s_*}')=(\ov\xi_{s+s_*},
\ov\xi_{s},\ov\eta_{s+s_*}, \ov\eta_{s}).
\end{array}\right.
\end{align}
Throughout this section, the indices $h,e,s$ have the ranges $1\leq e\leq e_*$, $e_*<h\leq e_*+h_*$, and $e_*+h_*<s\leq p-s_*$.
Notice that we can always normalize $\rho_0$ into the above normal form $\rho$.

  For a general  $M$ that  is a third-order perturbation of product quadric $Q$ and
  satisfies  condition~J.  
    define $\sigma_j=\tau_{1j}\tau_{2k_j}$. In suitable coordinates, $T_{ij}$
  (resp. $S_j$) is the linear part of $\tau_{ij}$ (resp. $\sigma_j$) at the origin.
  We say that the complex tangent of
  $M$ of a product quadric
  at the origin is of {\it abelian type},  if
  $\sigma_1, \dots, \sigma_p$
  commute pairwise.
 If each linear part  $S_j$ of $\sigma_j$ has exactly two eigenvalues $\mu_j,
\mu_j^{-1}$ that are different from $1$,  then $\cL S:=\{S_1,\ldots, S_p\}$ is of Poincar\'e type if and only if $|\mu_j|\neq1$ for all $j$.

  As an application of \rt{lFba}, we have the following convergent normalization.
  \begin{thm}[\!\!\cite{GS16}] \label{iabelm}
Let $M$  be a germ of real analytic submanifold  in $\cc^{2p}$
that is a third order perturbation of a product quadric given by \rea{Qgs2}-\rea{0ge1} with an abelian CR singularity. Suppose that $M$ has all eigenvalues of modulus different from one, i.e. it has no hyperbolic component $(h_*=0)$ while each $\gaa_s$ in \rea{Qgs2+} satisfies $\RE\gamma_s<1/2$ additionally.
Then $M$ is holomorphically equivalent to
\begin{gather}\nonumber 
\widehat M\colon
z_{p+j}=\Lambda_{1j}(\zeta)\zeta_j,\quad \Lambda_{1j}(0)=\la_j,\quad 1\leq j\leq
p,
\end{gather}
where  $\zeta=(\zeta_1,\ldots, \zeta_p)$ are the holomorphic solutions to
\begin{align*}\nonumber
\zeta_e&=
 A_e(\zeta)|z_e|^2-
 B_e(\zeta)(z_e^2+\ov z_e^2),\quad 1\leq e\leq
e_*,\\
\zeta_s&=
 A_s(\zeta)z_s\ov z_{s+s_*}-
 B_s(\zeta)(z_s^2+\Lambda_{1s}^2(\zeta)\ov z_{s+s_*}^2),\quad  e_*< s\leq e_*+s_*,\\
\zeta_{s+s_*}&= 
A_{s+s_*}(\zeta) \ov z_s z_{s+s_*}-
B_{s+s_*}(\zeta)(z_{s+s_*}^2+\Lambda_{1(s+s_*)}^2(\zeta)\ov z_{s}^2),
\nonumber
\end{align*}
 while 
the holomorphic function $ \Lambda_{1j}$ satisfies  $\Lambda_{2j}=\Lambda_{1j}^{-1}$ and
\begin{gather} 
	\label{lam1e}
	\Lambda_{1e}=\overline{\Lambda_{1e}\circ{\rho_z}},\  1\leq e\leq e_*; \quad
	\Lambda_{1s}^{-1}=\overline{\Lambda_{1(s+s_*)}\circ {\rho_z}},\ e_*<s\leq p-s_*,\\
	\label{rhoz5}
	{\rho_z}\colon\zeta_e\to\ov\zeta_e,
	\quad\zeta_s\to\ov\zeta_{s+s_*}, \quad\zeta_{s+s_*}\to\ov\zeta_s.
	\end{gather}
and $A_j,B_j$  are rational functions in $\Lambda_{1j}(\zeta)=\la_j+O(\zeta)$ $(1\leq j\leq p)$ defined by~:
\begin{gather} \label{AeAj}
	A_e=\frac{1+\Lambda_{1e}^{2}}{(1-\Lambda_{1e}^2)^{2}},\quad  A_j=\f{\Lambda_{1j}+\Lambda_{1j}^3}{(1-\Lambda_{1j}^2)^2}, \ j=s,s+s_*, \\
	B_j=\frac{\Lambda_{1j}}{(1-\Lambda_{1j}^2)^{2}}, \quad j=e,s,s+s_*. \label{AeAj+}
\end{gather} 
In particular, $\widehat M$ is contained in $z_{p+e}=\ov z_{p+e}$ and  $z_{p+s}\tilde\Lambda_{1s}^{2}(z'')=\ov z_{p+s+s_*}$, where $\zeta_j=z_{p+j}\tilde\Lambda_{1j}(z''), 1\leq j\leq p,$ is the inverse mapping of $z_{p+j}=\zeta_j\Lambda_{1j}(\zeta),1\leq j\leq p$.
\end{thm}

There are many non-product real submanifolds  of abelian CR singularity.
\begin{exmp}[\!\!\cite{GS16}]  Let $0<\gaa_i<\infty$.
Let $R(z_1,\ov z_1)=|z_1|^2+\gaa_1(z_1^2+\ov z_1^2)+O(3)$ be a real-valued power series in $z_1,\ov z_1$ of real coefficients. Then the origin is an abelian CR singularity of
$$
M\colon z_3=R(z_1,\ov z_1), \quad z_4=(z_2+2\gaa_2 \ov z_2+z_2z_3)^2.
$$
\end{exmp}

 As an application of \rt{iabelm},
we have the following flattening result~:
\begin{cor} Let $M$ be as in \rta{iabelm}. In suitable holomorphic coordinates, $M$
is contained in the linear subspace   defined by $z_{p+e}=\ov z_{p+e}$ and $z_{p+s}=\ov z_{p+s+s_*}$ where $1\leq e\leq
e_*$ and $ e_*< s\leq e_*+s_*$.
\end{cor}


\subsection{Analytic hull of holomorphy} 
As another application of the above normal form, we can  construct the local hull of holomorphy of $M$
via higher dimensional non-linear analytic polydiscs.
\begin{cor}[\!\!\cite{GS16}]  \label{andiscver}
Let $M$ be as in \rta{iabelm}.  Suppose that $M$
  has   only elliptic component of complex tangent. Then   in suitable holomorphic coordinates,
$\cL H_{loc}(M)$,  the local hull of holomorphy of $M$, is filled
by a real analytic family of analytic polydiscs of dimension $p$.
\end{cor}
The hulls of holomorphy for real submanifolds with a CR singularity have been studied extensively, starting with the work of Bishop. In the real analytic case with minimum complex tangent space at an elliptic
complex tangent, 
we refer to Moser-Webster~\cite{MW83} for $\gaa>0$, and Krantz-Huang \cite{HK95} for $\gaa=0$.  For the smooth case, see Kenig-Webster~\cite{KW82,KW84}, Huang~\cite{H98}.
For global
results on hull of holomorphy, we refer to \cite{bedford-gaveau, bedford-kling}.

\subsection{Rigidity of quadrics}We now study the
 rigidity problem 
 of a quadric under higher order analytic perturbations,
 i.e.  the problem if such a perturbation  remains holomorphically equivalent to the quadric if  it is   formally equivalent
 to the quadric.
  The rigidity problem is  reduced to a theorem of holomorphic linearization of one or several commuting diffeomorphisms   that was devised in \cite{stolo-bsmf}.
  We emphasize that condition (1.5) ensure that $M$ is diagonalizable (condition J). 
Using  a theorem on linearization of a family of commuting holomorphic mappings
in \cite{stolo-bsmf}, we proved  the following. 
\begin{thm}[\!\!\cite{GS16}] \label{quad-brjuno}
 Let $M$ be a germ of analytic submanifold that is an higher order perturbation of a product quadric $Q$ in $\cc^{2p}$ given by \rea{Qgs2}-\rea{0ge1}. Assume that $M$ is formally equivalent to $Q$.
 Suppose that  each hyperbolic component has an eigenvalue $\mu_h$ which
is either a root of unity or satisfies Brjuno condition, and each complex component has an eigenvalue $\mu_s$ that  is not  a root of unity or  satisfies the Brjuno condition.
Then $M$ is holomorphically equivalent to the product quadric.
\end{thm}
Here one says that $\la$ is a Brjuno number if $\la$ is not a root of unity, $|\la|=1$, and
\eq{brjuno-cond}
-\sum_{k=1}^\infty 2^{-k}\log\min_{1\leq j\leq 2^k}|\la^j-1|<\infty.
\eeq

%
%

\subsection{Attached complex submanifolds}
We now study the existence of
 holomorphic  submanifolds attached to the real submanifold $M$. These are complex submanifolds of dimension $p$ intersecting $M$ along two totally real analytic submanifolds that intersect  transversally at a CR singularity.  Attaching complex submanifolds has less constraints than
 finding a convergent normalization.
 A remarkable feature
of   attached complex submanifolds is that their existence depends only on the existence of suitable (convergent)
 invariant submanifolds of $\sigma$.
We now describe  convergent results for attached complex submanifolds.  The
results are for a   general
 $M$, including the one of which the complex tangent might not be of abelian type.

We say that a formal complex submanifold $K$ is  {\it attached} to $M$ if $K\cap M$
 contains at least two germs of  totally real and formal submanifolds $K_1, K_2$
that intersect
 transversally at a given CR singularity. In~\cite{Kl85}, Klingenberg showed that when
$M$ is non-resonant and $p=1$, there is   a unique   formal holomorphic
curve attached to $M$ with a hyperbolic complex tangent. He also proved the convergence of the attached formal holomorphic curve under
a Siegel small divisors
 condition.  When $p>1$, we  showed that generically there is no formal complex submanifold that can be attached
to $M$   if the CR singularity has
an elliptic component.

By adapting Klingenberg's proof for $p=1$
and using a theorem of P\"oschel~\cite{Po86}, the following is proved.
\begin{thm}[{\!\!\cite[Thm.~7.5]{GS16}}] \label{1conv}
Let $M$ be a germ of analytic submanifold that is an higher order perturbation of a product quadric $Q$ in $\cc^{2p}$ without elliptic components.
Assume that the eigenvalues $\mu_1,\ldots, \mu_p$,  $\mu_1^{-1}, \ldots, \mu_p^{-1}$ of $\sigma'(0)$ are distinct. Let $\e_h^2,\e_s^2=1$, $\nu_h:=\mu_h^{\e_h}$, $\nu_s:=\bar\mu_s^{\e_s}$ and $\nu_{s+s_*}:= \bar\nu_s^{-1}$.
Assume $\nu=(\nu_1,\ldots,\nu_p)$ is 
Diophantine in the sense of P\"oschel
~\cite{Po86}.
Then $M$ admits an attached complex submanifold $M_{\e}$. 
\end{thm}

Here the  Diophantine condition for $\nu$ in the sense of P\"oschel~\cite{Po86} is defined by 
$$
-\sum 2^{-k}\log\om_{\nu}(k)<\infty,\quad 
\om_{\nu}(k):=\min_{1< |P|<2^k,P\in\nn^p}\min_{1\leq i\leq p}\{|\nu^P-\nu_i|,|\nu^P-\nu_i^{-1}|\}.
$$

Finally, we proved  the convergence of {\it all} attached  formal 
submanifolds: 
\begin{thm}[{\!\!\cite[Thm.~7.6]{GS16}}] \label{allconv}  
Let $M$ be as in \rta{1conv}. 
Suppose that the $2p$ eigenvalues $\mu_1,\dots, \mu_p,\mu_1^{-1},\dots, \mu_p^{-1}$ of $\sigma$ are non-resonant.
 If 
  the eigenvalues of $\sigma$ satisfy a Bruno type condition or $M$ has pure complex CR singularity,
 all attached formal submanifolds are convergent.
\end{thm}

For the Brjuno type condition in \rt{allconv}, see ~\cite[(7.38)]{GS16}, which is condition  introduced in~\cite{stolo-bsmf} for linearization on ideals; the proof of \rt{allconv} also uses a result in~\cite{stolo-bsmf}. 


%
%
%

\subsection{Existence of divergent Poincar\'e-Dulac type normal forms}
Recall that the Moser-Webster formal normal form for the surface case is always convergent.
The $\Lambda_{11}, \dots,$ $ \Lambda_{1p}$ in \rt{iabelm} are subjected to further normalization. 
In~\cite{part2}, we find a  unique holomorphic normal form by refining the above normalization for $M$ satisfying a non resonance condition and a third order non-degeneracy condition (see \cite[Theorem 5.6]{part2}); in particular, it shows the existence of infinitely many formal invariants and non-product structures of the manifolds when $p>1$. For $p=1$ and $n>2$, the Moser-Webster normal form is also subject to further normalization. Under a third-order non-degenerate condition, the normal form is algebraic given by $x_n=z_1\bar z_2+(\gamma+x_2+\delta x_n^s)(z_1^2+\ov z_1^2)$ with $y_2=\cdots=y_{n}=0$.

The following result shows that not only a transformation that realizes a normal form can diverge, a normal form itself can diverge.
\begin{thm}[\!\!\cite{part2}] \label{divsig} There exists a non-resonant real analytic submanifold $M$ with
pure elliptic complex tangent in $\cc^6$ 
such that if its corresponding $\sigma$  is transformed into a map $\sigma^*$ 
which commutes with the linear part of $\sigma$ at the origin,
then $\sigma^*$ must diverge. 
\end{thm}
A transformation $\sigma^*$ that commutes with its linear part is called a Poincar\'e-Dulac normal form.
Recall that the existence of divergent Birkhoff normal form for real analytic Hamiltonian systems was proved  in~\cite{Gon12} by using small-divisors, which answers a question raised by Eliasson~\cite{El90}.  Prior to the work~\ci{Gon12}, P\'erez-Marco~\ci{PM03} proved that {\it if} there exists one divergent Birkhoff normal form, then divergent Birkhoff normal forms exist generically for the real analytic Hamiltonian systems that have the same the leading quadratic part. Later using the small-divisors, Yin~\ci{Yi15} also showed the existence of divergent Birkhoff normal form for real analytic area-preserving mapping with an elliptic fixed point. Recently,  the existence of divergent Birkhoff normal forms for Hamiltonian systems and area-preserving at the absence of small-divisors has been solved affirmatively; Krikorian~\ci{Kr22}  settled the case for the real analytic symplectic mappings and Fayad~\ci{Fa23} settled the case for the real analytic Hamiltonian systems. The reader is referred to ~\cite{PM03, Gon12, Yi15, Kr22, Fa23} for references.
The proof of \rt{divsig} also relies on the small divisors generated by the eigenvalues $\mu_1,\mu_2,\mu_3,\mu_1^{-1},\mu_2^{-1},\mu_3^{-1}$. Although all $|\mu_j|>1$, $\mu_1^a\mu_2^b\mu_3^c-1$ can still tend to $1$ rapidly. \rt{divsig} has a new feature that cannot appear in Birkhoff normal forms. 
It is well-known that Birkhoff normal forms for the non-resonant Hamiltonian systems and symplectic mappings are unique up to a perturbation of coordinates. However, the Poincar\'e-Dulac normal forms of $\sigma$  are still subject to further normalization under a large group of coordinate changes; in fact it is easy to produce a divergent Poincar\'e-Dulac normal form because of the freedom of changing coordinates. The new feature of \rt{divsig} is the divergence of all Poincar\'e-Dulac normal forms of a single $\sigma$.

\section{Some open problems}\label{sect:openprob}
\setcounter{thm}{0}\setcounter{equation}{0}

We have surveyed some results on normal forms related to CR singularity. They are of course many important topics which  we have not discussed. Among the results we stated, it is clear that some of these results give arise to new questions. We conclude this survey with a few open problems. 

\subsection{Convergence of normal forms for real analytic surfaces with $\gamma=0$}
\begin{problem}Let $M$ be a real analytic surface in $\cc^2$ with $\gamma=0$. Determine if the Huang-Yin  normal form for $M$ \rea{HYnf} converges. Determine if the normal form \rea{vol-nf} for $M$ under formal under volume-preserving maps converges.
\end{problem}
As \rt{divsig} relies on small-divisors, we have the following.
\begin{problem}Let $\mu_1,\cdots,\mu_p$ be complex numbers. Assume that $\mu=(\mu_1,\dots,\mu_p)$ are non-resonant, i.e. $\mu^P\neq1$ for $P\in\zz^p$ unless $P=0$.  Determine if there exists a real analytic $(2p)$-submanifold with $p>1$ in $\cc^{2p}$ that admits maximum number of deck transformations such that the eigenvalues of $\sigma$ for $M$ are $\mu_1,\dots,\mu_p,\mu_1^{-1},\dots,\mu_p^{-1}$ and all Poincaré-Dulac normal forms of $\sigma$ are divergent.
\end{problem}
\subsection{Optimal condition on equivalence to Bishop quadric $Q_\gamma$ with $\gamma>1/2$}
\begin{conj}Let $\mu$ be a complex number  that is not a root of unity and $|\mu|=1$. If $\mu$  is not a Brjuno number defined by \re{brjuno-cond},
there exists a real analytic $M$ in $\cc^2$ that is formally but not holomorphically equivalent to $Q_{\gamma}$ with Moser-Webster invariant $\mu$. 
 \end{conj}
\rt{quad-brjuno} says that the Brjuno condition and formal equivalence to the quadric are sufficient for the holomorphic equivalence to the quadric. It is a well-known theorem of Yoccoz~\ci{Yo95} that the Brjuno condition on $\la$ is necessary for the holomorphic linearization of the one dimensional map $z\to \la z+z^2$. 
The sufficiency of the Brjuno condition was established by Brjuno~\ci{BR71}. To reader is referred to~\ci{Gon04} for a connection between a special family of the one-dimensional maps and hyperbolic surfaces that are formally equivalent to the quadric.

\subsection{Classification of $2p$-submanifolds in $\cc^{2p}$}

Consider  germs of $2p$-dimensional real analytic submanifolds $M$ in  $(\cc^{2p},0)$ of the form
\eq{mzpjintr+}
M\colon z_{p+j}=E_j(z',\ov z'),
\quad 1\leq j\leq p,
\eeq
where $z'=(z_1,\ldots, z_p)$ and
\eq{callq}
E_j(z',\ov z')=h_j(z',\ov z')+q_j(\ov z') + O(|(z',\ov z')|^3), \quad  
j=1,\dots, p.
\eeq
Moreover,    $h_j(z',\ov z')$ are homogeneous quadratic polynomials in $z',\ov z'$ without holomorphic or anti-holomorphic terms,
and  $q_j(\ov z')$ are homogeneous quadratic polynomials in $\ov z'$. 
As a general motivation, we would like to pose the following~:
\begin{problem}For $p>1$, classify all real analytic $(2p)$-submanifolds of $\cc^{2p}$ of the form~\rea{mzpjintr+}-\rea{callq} that satisfy condition B, i.e. the   mapping $q=(q_1,\dots, q_p)$ in \rea{callq}  satisfies $q^{-1}(0)=\{0\}$.
\end{problem}
Some more  restricted versions are still very challenging~:
	\begin{problem}
		For $p>1$, let $M\subset \cc^{2p}$ be a    real analytic $(2p)$-submanifolds of the form~\rea{mzpjintr+}-\rea{callq} that satisfies condition D, i.e there is a maximum number of deck transformations. Assume further that the quadratic parts of $M$ is a {\it product of hyperbolic quadrics}, i.e.
$$
h_j(z',\bar z')+q_j(\bar z')=z_j\bar z_j+\gamma_j(z_j^2+\bar z_j)^2, \quad 1\leq j\leq p,\quad  1/2<\gamma_j<\infty.
$$
Find conditions so that a KAM phenomena holds.
	\end{problem}
	\begin{problem}
		For $p>1$, classify real analytic $(2p)$-submanifolds $M$ of $\cc^{2p}$ that satisfy condition D and that are higher-order perturbations of a product of quadrics of different types, e.g. one elliptic and one hyperbolic or parabolic.
	\end{problem}
	\begin{problem}
		For $p>1$, consider a real analytic $(2p)$-submanifold $M$ of $\cc^{2p}$ which is a higher-order perturbation of a {\it non-product} of quadrics of the form
$$
Q_{h,q}\colon z_{p+j}= h_j(z',\ov z')+q_j(\ov z'),\quad 1\leq j\leq p.
$$
 Assume that $M$ is formally equivalent to $Q_{h,q}$. 
 Determine if it is holomorphically equivalent to the $Q_{h,q}$.
\end{problem}

\subsection{Classification of Levi-flat submanifolds with CR singularity}
Gong-Lebl showed the following.
\begin{thm}[\!\!\cite{GL15}]
 Let $M\subset\cc^{n+1}$ with $n\geq 2$ be a germ of real-analytic real codimension-two submanifold with CR singularity at the origin, given by
 $$
 w=A(z,\ov z)+B(\bar z,\bar z)+O(3)
 $$
 for homogeneous quadratic polynomials $A$ and $B$ with $A$ or $B$ not identically zero. Suppose that $M$ is Levi-flat away from its CR singular set.  Then the quadric defined by $w=A(z,\bar z)+B(\bar z,\bar z)$ is also Levi-flat. Further the quadric is holomorphically equivalent to exactly one of the following Levi-flat quadrics
 \bpp
 \item $A.q\colon w=\bar z_1^2+\cdots+\bar z_k^2$ with $1\leq k\leq n$.
 \item $B.\gamma\colon w=|z_1|^2+\gamma \bar z_1^2$ with $0<\gamma<\infty$.
 \item $C.0\colon w=\bar z_1z_2$, or $C.1\colon w=\bar z_1z_2+\bar z_1^2$.
 \epp
 \end{thm}
 There are some partial results~\cite{GL15} for normalization  for higher order perturbations of the quadrics. However, both the formal and holomorphic classifications are largely open.
 
 \subsection{Hull of holomorphy for maximum CR singularity}
 One of very first questions asked by Bishop is the hull of holomorphy of smooth but not necessarily real analytic  $n$-submanifolds in $\cc^n$. This question has been studied extensively from local and global points of view. 
 
 To be clear, let us define the following.
 \begin{defn}\label{hull-defn}Let $K$ be a compact subset of $\cc^n$. $Hull(K)$ is the intersection of all pseudoconvex domains in $\cc^n$ that contain $K$. The polynomial hull $P(K)$ of $K$ is the set of points $z\in\cc^n$ such that $|p(z)|\leq\sup_{w\in K}|p(w)|$ for all holomorphic polynomials $p$. 
 \end{defn}
 For a submanifold $M$ in $\cc^n$, 
 we can also seek local hull of holomorphy of $M$ or local polynomial hull of $M$ by treating $M$ as a germ at a point, in which case we ask for a description of $Hull(M_t)$ or $ P(M_t)$ for a choice of open subsets $M_t$  of  $M$, where $t$ belongs to some parameter $T$ and $\cap_{t\in T} M_t=\{x\}$. We then say that $M$ has non trivial local hull of holomorphy if $Hull(M')\setminus M'$ is non-empty for any neighborhood $M'$ of $x$ in $M$.

 Bishop showed that a smooth real surface in $\cc^2$ has non-trivial local hull of holomorphy at an elliptic complex tangent. We have also discussed its further development in  Sections 2, 8 and 10.   Forstneri\v{c} and Stout~\ci{FS91} proved that near a hyperbolic complex tangent a real surface in $\cc^2$ is polynomial convex. The situation with parabolic complex tangent is more complicated. When $M$ defined by $z_2=f(z_1)$  has an isolated CR singularity at the origin one can define the index of the complex tangent to be the winding number of $\partial f(z_1)/{\partial \bar z_1}$ at $z_1=0$. Wiegerinck~\ci{Wi95} proved that $M$ is never locally polynomial convex at a parabolic complex tangent with index $1$; for index $-1$, J\"oricke~\ci{Jo97} showed that $M$ is always locally polynomial convex. For the index $0$ case, she  showed  that an $C^2$ surface $M$ is either locally polynomial convex or  $K\cap\overline{(\hat K\setminus K)}$ has the structure of an ``onion", and both cases occur.
 
 Our result \nrc{andiscver} on local hull of holomorphy filled by a family of polydiscs relies heavily on the convergence of the normal form of  abelian CR singularity with elliptic type.  Therefore, the theory of hull of holomorphy remains to be developed.
 Here we single out a simple but unsettled case.
 \begin{problem}Let $M\subset\cc^4$ be a $4$-submanifold defined by
 	$$
 	z_3=z_1\ov z_1+\gamma_1(z_1^2+\ov z_1^2)+R_1(z_1,z_2), \  z_4=z_2\ov z_2+\gamma_2(z_2^2+\ov z_2^2)+R_2(z_1,z_2),\  0\leq\gamma_i<1/2,
 	$$
 	where $R_i(z_1,z_2)=O(3)$ are $C^\infty$ or $C^\om$ complex varied functions. Determine if $M$ has a non-trivial hull of holomorphy at the origin.
 \end{problem}
 Recently, Modal~\ci[Thm.~1.8]{Mo23} showed the following: Let
 $$
 \Gamma\subset\cc^{2p}: w_{p+j}=\bar z_j^{m_j}+R_j(z), \quad 1\leq j\leq p, \quad z\in\ov D,
 $$
 where $m_j$ are positive integers,  $D$ is a polydisk in $\cc^p$,  $R_j$ are functions on an open set $\Omega$ containing $\ov D$, and
 $$
 \sum_{j=1}^p|R_j(\tilde z)-R_j(z)|^2\leq c\sum_{j=1}^p|\tilde z_j^{m_j}-z_j^{m_j}|^2 \quad c\in[0,1)
 $$
 for all $\tilde z,z\in\Om$. Then $\Gamma$ is polynomial convex. Note that the last condition is equivalent to $R(z)=\tilde R(z_1^{m_1},\dots, z_p^{m_p})$ with $\tilde R$ having Lipschitz ratio less than $1$.

 \subsection{Topological invariants of  submanifolds with maximum CR singularity}
 
 One of very first questions asked by Bishop is the global topological invariants of compact real submanifolds in a complex manifold. He observed that the standard $2$-sphere in $\cc^2$ has at least exactly elliptic complex tangents. There are many other results in this direction; for instance, see Lai~\ci{La72} and Forstneri\v{c}~\ci{Fo92}, and reference therein. However, the global topological invariants for maximum CR singularity remains to be developed. 
 
 \subsection{Classifications 
  under volume-preserving transformations}
 
 One can classify real analytic submanifolds by using unimodular holomorphic mappings, i.e. the mappings that preserve holomorphic $n$-form $dz_1\wedge\cdots\wedge dz_n$.  This has been achieved by Gong real analytic $n$ submanifolds in $\cc^n$ with minimum CR singularity with an elliptic or hyperbolic complex tangent. The case $\gamma=0$ is largely open except the case that the real analytic submanifold is formally equivalent to $Q_0$ in $\cc^2$ as discussed in Section~\ref{gamma=0}. The case $\gamma=1/2$ is also largely open except when $M$ is formally equivalent to $Q_{1/2}$ as discussed in Section~\ref{gamma=.5}.
 In particular, we have the following.
 
 \begin{problem}Classify $n$-dimensional real analytic submanifolds in $\cc^n$ of which $\dim M_{CRs}$ can be any value in $\{0,\dots, n-2\}$ with Bishop invariant $1/2$. 
 	Same classification problem but using unimodular mappings.
 	In particular,  classify real analytic submanifolds in $\cc^n$ that is formally equivalent to
 	\begin{equation}\label{delta2}
 		Q^*_{1/2}\colon z_n=q_{1/2}(z_1,\bar z_1)+i x_2(z_1+\bar z_1), \quad y_2=\cdots y_{n-1}=0, \quad n>2.
 	\end{equation}
 \end{problem}
 Via the Moser-Webster theory, the classification for $M$ that is formal equivalent to $ Q^*_{1/2}$ is closely related to the following.
 \begin{problem}Classify biholomorphic mappings on $\cc^n$ ($n>2$) that are formally equivalent to
 	$$
 	\hat \sigma(\xi,\eta,\zeta)= (\xi+4i\zeta_2,\eta+4\xi-8i\zeta_2,\zeta),\quad (\xi,\eta)\in\cc^2, \ \zeta=(\zeta_2,\dots,\zeta_{n-1})\in\cc^{n-2}.
 	$$
 \end{problem}
 Note that the set of fixed points of $\hat\sigma$ has codimension $2$ in $\cc^n$ and Voronin's result does not apply directly. 
 Of course, one can also consider $n$-manifolds in $\cc^n$ that are formally equivalent to  product quadrics of some or all $Q_{\gamma_e},Q_{\gamma_h}, Q_{1/2}, Q_{1/2}^*, Q_0$.
 
 \subsection{Classification of (weakly) reversible holomorphic mappings}
 We say that a biholomorphic mapping $g$ is weakly reversible, if $g^{-1}=h^{-1}gh$ for some biholomorphic map $h$ that is not necessary an involution, and $h$ is called  a reverser.
 The following two  normal forms were proved by O'Farrell-Zaitsev~\ci{OZ14}.
  \begin{thm}[\!\!\cite{OZ14}] A formal map $g$ of $\cc^2$ with linear part $(z_1,z_2)\to(\lambda z_1,\lambda^{-1} z_2)$ for $\lambda$ not a root of unity is formally weakly  
   reversible if and only if it is formally linearizable, or formally equivalent to one of the following maps
 \begin{equation}\label{two-nf}
 \left(z_1,z_2\right)\to \left(\lambda (1+p^k)z_1,\frac{z_2}{\lambda(1+p^k)}\right), \left(z_1,z_2\right)\to \left(\frac{\lambda (1+p^k)^{\frac{1}{k}}}{(1+2p^k)^\frac{1}{k}}z_1,\frac{z_2}{\lambda(1+p^k)}\right)
 \end{equation}
 with $p=z_1z_2$ and $k=1,2,\dots$. Furthermore, if $g$ is convergent,  $|\la|\neq 1$ and the formal normal form is given by the first one, the normalization converges. 
 \end{thm}
They also find reversers for the two normal forms and the only the first normal form admits an reverser that is an involution.

\begin{problem} 
Determine if the second normal form in \re{two-nf} can  be achieved by a convergent transformation when $|\la|\neq1$.  It seems that the answer is unknown for the case $|\lambda|=1$ with $\la$ not a root of unity.
  Find the formal and holomorphic classifications for weakly reversible holomorphic mappings  if $\lambda$ is a root of unity. 
\end{problem}

%

\appendix
\section{Overview of the 
 \'Ecalle--Voronin
 theory}\label{ecalle-voronin}

To motivate our results, let us briefly recall some of the basics of analytic theory of  parabolic diffeomorphisms of $(\cc,0)$.
The idea behind this overview is to present afterwards a somewhat similar approach to our problem in dimension 2.
For more details on the \'Ecalle-Voronin theory see \cite{ecalle2,voronin, malgrange-bourbaki, EIlSV, AG05}  in which tangent to identity diffeomorphisms are considered. For general case, see \cite{martinet-ramis2}.
Our exposition follows  \cite{stolo-klimes} closely.

Let $\phi(z)=\lambda z+\hot(z)\in\Diff(\cc,0)$ be a germ of analytic diffeomorphism fixing the origin in $\cc$,
where $\lambda=e^{2\pi i\frac{q}{p}}$ is a root of unity of some order $p\geq 1$,  $\lambda^p=1$.
Its $p$-th iteration $\phi^{\circ p}(z)=z+\hot(z)$ is a diffeomorphism tangent to the identity, and as such it possesses a unique formal infinitesimal generator: a formal vector field $\hat{\bX}(z)$ at $0$ with vanishing linear part, such that the Taylor series of $\phi^{\circ p}(z)$ agrees with the formal time-1-flow $\exp(\hat{\bX})(z)$ of $\hat{\bX}$.
The formal vector field $\hat{\bX}$ can be conjugated by some formal tangent-to-identity map $\hat{\Psi}(z)\in\hatDiff_{\id}(\cc,0)$ to its {\it normal form}
\begin{equation*}
	{\bX}_{\rm nf}(z)=\frac{c\,z^{kp}}{1+c\,\mu\, z^{kp}}\, z\frac{\partial}{\partial z},\quad k\geq 1,\quad c\neq 0,	
\end{equation*}
which is invariant by the rotation $z\mapsto\lambda z$. Consequently also the germ $\phi(z)$ is formally conjugated to the \emph{normal form}
\begin{equation*}
	\phi_{\rm nf}(z)=\lambda\exp(\tfrac1p{\bX}_{\rm nf})(z).
\end{equation*}	
As it turns out, while the formal conjugacy $\hat\Psi(z)$ is generically divergent, it is Borel summable (of order $kp$) on sectors.

\begin{thm}[Birkhoff~\ci{Bi39}, Kimura\cite{kimura}, \'Ecalle~\cite{ecalle2}, Voronin~\cite{voronin}, \dots]\label{thm:kimura}
	The germ $\phi(z)$ is conjugated to its  normal form $\phi_{\rm nf}(z)$ by a {\it cochain of bounded analytic transformations}
	$\big\{\Psi_{\Omega_j}(z)=z+\hot(z)\big\}_{j\in\zz_{2kp}}$
	on a covering by $2kp$ sectors (Leau--Fatou petals) $\Omega_j$, $j\in\zz_{2kp}$,\footnote{The sectorial covering is $\lambda$-invariant: writing $\lambda=e^{2\pi i\frac{q}{p}}$ then for every sector $\Omega_j$ the rotated sector $\lambda\Omega_j=\Omega_{l}$, $l=j+2kq\mod 2kp$, belongs again to the covering.}
	\begin{equation*}
		\Psi_{\lambda\Omega_j}\circ\phi(z)=\phi_{\rm nf}\circ\Psi_{\Omega_j}(z),\quad z\in\Omega_j.
	\end{equation*}	
	Such normalizing cochain $\big\{\Psi_{\Omega_j}(z)\big\}_{j\in\zz_{2kp}}$ is \emph{unique up to} left composition with cochains $\big\{\exp(C_{\Omega_j}{\bf X}_{\rm nf})(z)\big\}_{j\in\zz_{2kp}}$, $C_{\Omega_j}\in\cc$, of flow maps of ${\bX}_{\rm nf}$.
\end{thm}

The form of these sectorial domains $\Omega_j$ (Figure~\ref{figure:parabolicpetals}) is related to the dynamics of ${\bX}_{\rm nf}$: they are spanned by the real-time trajectories of the family of rotated vector fields $e^{i\theta}{\bX}_{\rm nf}(z)$,
i.e. by the real curves
\[\frac{dz}{dt}=e^{i\theta}\frac{c\,z^{kp}}{1+c\,\mu\, z^{kp}}z,\qquad t\in\rr,\]
that stay inside some disc $\{|z|<\delta_1\}$, where $\theta$ is allowed to vary in some interval $]\delta_3,\pi-\delta_3[$,
for some $\delta_1,\delta_3>0$. See Figure~\ref{figure:parabolictrajectories}.

\begin{figure}[t]
	\centering	
	\begin{subfigure}{.31\textwidth} \includegraphics[angle=90, width=\textwidth]{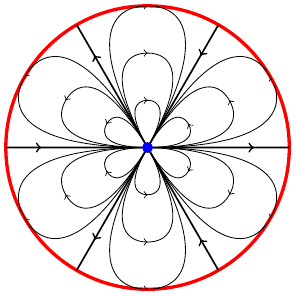} \caption{} \label{figure:parabolictrajectories} \end{subfigure}	
	\qquad
	\begin{subfigure}{.4\textwidth} \includegraphics[width=\textwidth]{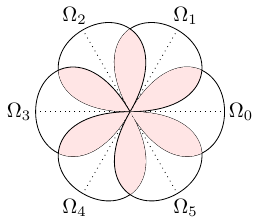} \caption{} \label{figure:parabolicpetals}  \end{subfigure}	
	\caption{(a) Real-time trajectories of the vector field $e^{i\theta}{\bX}_{\rm nf}$ inside a small disc.
		(b) The Leau--Fatou petals $\Omega_j$, $j\in\zz_{2kp}$.}
	\label{figure:parabolic}
\end{figure}

\medskip
The equivalence class of the set of the $2kp$ \emph{transition maps}
\[\psi_j=\Psi_{\Omega_{j-1}}\circ\Psi_{\Omega_j}^{\circ(-1)}\quad \text{on the {\it intersections} }\ \Omega_{j-1}\cap\Omega_j,\quad j\in\zz_{2kp},\]
modulo conjugation by cochains of flow maps $\big\{\exp(C_{\Omega_j}{\bX}_{\rm nf})\big\}_{j\in\zz_{2kp}}$, $C_{\Omega_j}\in\cc$,
\[\psi_{j}\simeq \exp(C_{\Omega_{j-1}}{\bX}_{\rm nf})\circ\psi_{j}\circ\exp(-C_{\Omega_{j}}{\bX}_{\rm nf}),\qquad j\in\zz_{2kp},\]
is then called a \emph{cocycle}. It is an analytic invariant of $\phi$  which expresses the obstruction to convergence of the formal normalizing transformation.
It was initially described by G.D.~Birkhoff \cite{Bi39} and later independently rediscovered by J.~\'Ecalle \cite{ecalle2} and S.M.~Voronin \cite{voronin}. The realization of moduli (assertion (2)) below was also conjectured by Birkhoff,  who proved a version of realization for a finite-jet deformation of realized  $\{\Psi_j\}_{j\in\zz_{2}}$; see~\cite[p.~83]{Bi39}.

\begin{thm}
	\begin{enumerate}
		\item \textnormal{(Birkhoff~\cite{Bi39}, \'Ecalle~\cite{ecalle2}, Voronin~\cite{voronin}).}
		Two germs $\phi$, $\phi'$ that are formally tangent-to-identity equivalent are analytically tangent-to-identity equivalent if and only if their cocycles $\big\{\psi_j\big\}_{j\in\zz_{2kp}}$, $\big\{\psi_j'\big\}_{j\in\zz_{2kp}}$ agree.
		
		\item \textnormal{(\'Ecalle~\cite{ecalle2}, Malgrange~\cite{malgrange-bourbaki}, Voronin~\cite{voronin}).}
		For each formal normal form $\phi_{\rm nf}$ and each collection of maps $\big\{\psi_j\big\}_{j\in\zz_{2kp}}$ on the intersections sectors, that are asymptotic to the identity and commute with $\phi_{\rm nf}$:
		\[\phi_{\rm nf}\circ\psi_j=\psi_{j+2kq}\circ\phi_{\rm nf},\]
		there exists an analytic map $\phi$ whose cocycle is represented by $\big\{\psi_j\big\}_{j\in\zz_{2kp}}$.
	\end{enumerate}	
\end{thm}

If one wants to obtain the modulus of analytic equivalence with respect to conjugation by general transformations in $\Diff(\cc,0)$, one has to consider the cocycles modulo an action of the group of rotations $z\mapsto e^{2\pi i\frac{r}{kp}}z$, $r\in \zz_{kp}$, which preserve ${\bX}_{\rm nf}$. 	

\medskip

We remark that the Écalle-Voronin moduli space has quite a few applications, see~\cite{EIlSV}. For instance it plays an important role in the classifications of two transversal intersecting real analytic curves in $\cc$ and cusp-type singular real analytic curves; see Nakai~\ci{Na98}. Ahern-Gong~\cite{AG05, AG05a}. In~\ci{Je08}, Jenkins   used normalization on sectorial domains to investigate the classification of biholomorphic mappings of the form
$$
(z,w_1,\dots, w_n)\to (f(z),\la_1w_1g_1(z),\dots, \la_nw_ng_n(z))
$$
where $f(z)=z+O(2), g_j(z)=\lambda_j+O(1)$ are holomorphic functions with $|\lambda_j|\neq0,1$.
 Sectorial holomorphic solutions has 
  been obtained as "holomorphic realization" of formal series solutions of holomorphic differential equations near an irregular singularity \cite{ramis-houches} and they were used to define the moduli of such equations \cite{babbitt-v}. It has been then used for germs of holomorphic vector fields at a fixed point \cite{martinet-ramis1,martinet-ramis2,camacho-sad-ln} in $(\cc^2,0)$. The moduli space obtained is related to that of Écalle-Voronin through the {\it holonomy of an invariant manifold}. There are many other applications of these concepts among which are the theory of $q$-difference equations \cite{rsz-q-diff} and classification of CR manifolds \cite{KLS}.

\bibliographystyle{alpha}
\def\cprime{$'$}

\end{document}